\documentclass[onefignum,onetabnum]{siamart171218}

\usepackage{lipsum}
\usepackage{amsmath}
\usepackage{amsfonts}
\usepackage{amssymb}
\usepackage{graphicx}
\usepackage{epstopdf}
\usepackage{algorithm,algpseudocode}
\algnewcommand{\algorithmicforeach}{\textbf{for each}}
\algdef{SE}[FOR]{ForEach}{EndForEach}[1]
  {\algorithmicforeach\ #1\ \algorithmicdo}
  {\algorithmicend\ \algorithmicforeach}
\usepackage{subcaption}
\usepackage{bbm} 
\usepackage{bm} 
\usepackage{cleveref} 
\usepackage{enumitem,array}
\usepackage{verbatim}
\usepackage{overpic}
\usepackage{mdframed}
\usepackage{tikz}
\usetikzlibrary{matrix}
\usepackage{algpseudocode}
\usepackage{todonotes}
\usepackage{cite}

\ifpdf
  \DeclareGraphicsExtensions{.eps,.pdf,.png,.jpg}
\else
  \DeclareGraphicsExtensions{.eps}
\fi

\usepackage{booktabs}
\usepackage{tabularx}

\newtheorem{remark}{Remark}

\newtheorem{example}{Example}

\DeclareMathOperator*{\argmin}{arg\,min}

\DeclareMathOperator*{\vol}{vol}

\newcommand{\abs}[1]{\left|#1\right|}

\newcommand{\br}{\mathbf{r}}

\newcommand{\state}{\bm{x}}
\newcommand{\statered}{\bm{x}_r}

\newcommand{\rob}{\bm{V}}

\newcommand{\R}{\mathbb{R}}

\newcommand{\cO}{\mathcal{O}}
\newcommand{\cC}{\mathcal{C}}

\usepackage{hyperref}
\definecolor{mylinkcolor}{RGB}{0,0,130}
\hypersetup{colorlinks,allcolors=mylinkcolor,citecolor=mylinkcolor}

\crefname{section}{Section}{Sections}
\crefname{subsection}{Subsection}{Subsections}

\headers{How to reveal the rank of a matrix?}{}

\title{How to reveal the rank of a matrix?}

\author{Anil Damle\thanks{Department of Computer Science, Cornell University, Ithaca, NY 14853 (\email{damle@cornell.edu}).} \and Silke Glas\thanks{Department of Applied Mathematics, University of Twente, Enschede, The Netherlands (\email{s.m.glas@utwente.nl}).} \and Alex Townsend\thanks{Department of Mathematics, Cornell University, Ithaca, 14853, NY, USA 
  (\email{townsend@cornell.edu}).} \and Annan Yu\thanks{Center for Applied Mathematics, Cornell University, Ithaca, 14853, NY, USA (\email{ay262@cornell.edu}).}}

\ifpdf
\hypersetup{
  pdftitle={pivoting},
  pdfauthor={}
}
\fi

\begin{document}

\maketitle

\begin{abstract}
We study algorithms called rank-revealers that reveal a matrix's rank structure. Such algorithms form a fundamental component in matrix compression, singular value estimation, and column subset selection problems. While column-pivoted QR has been widely adopted due to its practicality, it is not always a rank-revealer. Conversely, Gaussian elimination (GE) with a pivoting strategy known as global maximum volume pivoting is guaranteed to estimate a matrix's singular values but its exponential complexity limits its interest to theory. We show that the concept of local maximum volume pivoting is a crucial and practical pivoting strategy for rank-revealers based on GE and QR. In particular, we prove that it is both necessary and sufficient; highlighting that all local solutions are nearly as good as the global one. This insight elevates Gu and Eisenstat's rank-revealing QR as an archetypal rank-revealer, and we implement a version that is observed to be at most $2\times$ more computationally expensive than CPQR. We unify the landscape of rank-revealers by considering GE and QR together and prove that the success of any pivoting strategy can be assessed by benchmarking it against a local maximum volume pivot.
\end{abstract}



\section{Introduction}
\label{sec:introduction}
Rank-revealers are algorithms that estimate a matrix's singular values, and they play an important role in numerical linear algebra and its applications. They are central to the development of least-squares solvers~\cite{chan1990computing,chan1992applications,golub1965numerical,lawson1995solving}, tensor approximation~\cite{grasedyck2013literature}, and interpretable low-rank approximations~\cite{dong2023simpler,xia2024making,mahoney2009cur}. They are a key part of the computational pipeline for problems in data-fitting~\cite{minden2017fast}, computational quantum chemistry~\cite{damle2015compressed,damle2018disentanglement}, reduced-order modeling~\cite{chaturantabut2010nonlinear,drmac2016new}, discretized integral equations~\cite{bebendorf2008hierarchical,ho2012recursive}, and data science~\cite{damle2019simple} (see~\cref{sec:Impact}). Attend any numerical linear algebra conference these days, and you will likely find a large fraction of the talks about matrix compression, singular value estimation, and column subset selection. Rank-revealing algorithms are a core subroutine for these tasks.

With all this attention on the subject, there is a wide variety of rank-revealing algorithms on the market---both deterministic and randomized---with a wide variety of theoretical guarantees and algorithmic costs. Some practitioners care about the quality of a low-rank approximation~\cite{chan1992applications}, others care about reasonable bases for the four fundamental subspaces of a matrix~\cite{bischof1992signal}, and others worry about computational or communication costs\cite{demmel2015communication,grigori2008communication}. In this zoo of techniques, finding a consistent definition of a rank-revealer isn't easy.  

We take the view that algorithms are rank-revealers, not factorizations, so our definition of rank-revealer is factorization-independent (see~\cref{eq:GoodLeadingSV,eq:GoodTrailingSV}). This gives us a standard definition on which to compare any algorithm that attempts to estimate the singular values of a matrix. We show that the only way for Gaussian elimination (GE) and QR to be rank-revealers and have their partial factorizations satisfy so-called interpolative bounds (see~\cref{def:InterpolativeBoundsQR,def:InterpolativeBoundsGE}) is for them to use a near-local maximum volume pivoting strategy (see~\cref{sec:Necessary}). 

Local maximum volume pivoting is a simple strategy to ensure that GE and QR efficiently compute reasonable singular value estimates and a near-optimal low-rank approximation. Rank-revealers can often replace an expensive singular value decomposition (SVD), leading to much faster algorithms across many disciplines without sacrificing theoretical statements. Moreover, the bases formed for the column and row space of a matrix by rank-revealers are often of great interest. 

\subsection{What is a rank-revealer?}
You might imagine that a rank-revealer is defined as a method that, when applied to any matrix $A$, yields both the rank, $k$, of $A$ and a representation of $A$ in low-rank form, i.e., $A= XY^\top$. Here, $X$ and $Y$ are matrices containing $k$ columns, and $Y^\top$ is the transpose of $Y$. However, such a definition is only useful in exact arithmetic as it is sensitive to perturbations. For example, ${\rm fl}(A)$ can have full rank even if $A$ has low rank, where ${\rm fl}(A)$ is the matrix $A$ rounded on a computer by floating-point. Hence, we are rarely interested in the exact rank of a matrix in most computational applications. 

Instead, we find it more useful to think of a rank-revealer as an algorithm for estimating a matrix's singular values.  Any square or rectangular matrix $A\in\mathbb{R}^{m\times n}$ has associated with it a sequence of nonincreasing singular values $\sigma_1(A)\geq \cdots \geq \sigma_{\min(m,n)}(A)\geq 0$. The Eckart–Young Theorem says that the singular values of $A$ tell us how well $A$ can be approximated by a rank $k$ matrix~\cite{eckart1936approximation}. That is, 
\begin{equation} 
\sigma_{k+1}(A) = \min_{{\rm rank}(B) \leq k} \|A - B\|_2, \qquad 1\leq k\leq \min(m,n) -1,  
\label{eq:EckartYoung} 
\end{equation} 
where $\|\cdot\|_2$ denotes the matrix $2$-norm. Most rank-revealers not only estimate a matrix's singular values but also construct a near-optimal rank $k$ approximation and bases for the column and row spaces of $A$. 

We call an algorithm a rank-revealer if there is a function $\mu_{m,n,k}$, bounded by a polynomial in $m$, $n$, and $k$, such that for any matrix $A\in\mathbb{R}^{m\times n}$, of any size, and any $1\leq k\leq \min(m,n)$, the algorithm computes a rank $\leq k$ matrix $A_k$ satisfying
\begin{equation} 
\frac{1}{\mu_{m,n,k}}\sigma_j(A) \leq \sigma_j(A_k)\leq \mu_{m,n,k} \sigma_j(A), \qquad 1\leq j\leq k,
\label{eq:GoodLeadingSV} 
\end{equation} 
and 
\begin{equation} 
\frac{1}{\mu_{m,n,k}}\sigma_{k+j}(A) \leq \sigma_j(A-A_k) \leq \mu_{m,n,k} \sigma_{k+j}(A),\qquad 1\leq j\leq \min(m,n)-k,
\label{eq:GoodTrailingSV}
\end{equation} 
where $\mu_{m,n,k}\geq 1$. Here, $\mu_{m,n,k}$ is only allowed to depend on $m$, $n$, and $k$, and not on $A$. The closer $\mu_{m,n,k}$ is to the value of $1$, the better the rank-revealer is at estimating the singular values of any matrix.  By~\cref{eq:EckartYoung}, we always have that $\sigma_{k+j}(A) \leq \sigma_j(A-A_k)$ so the lower bound in~\cref{eq:GoodTrailingSV} is trivially satisfied.

The best rank-revealer for minimizing $\mu_{m,n,k}$ is the truncated SVD (the full SVD of $A$ truncated after $k$ rank one terms), which achieves $\mu_{m,n,k} = 1$. However, one does not measure the success of a rank-revealer by only using the value of $\mu_{m,n,k}$. Instead, other factors come into play, such as (a) the computational cost of computing $A_k$, (b) the interpretability of the column and row space of $A_k$, and (c) structure-preserving features of $A_k$. These factors make GE and QR style rank-revealers important as they are more efficient than the SVD, can give more interpretable low-rank approximants, and preserve some matrix structures. Rank-revealers are also used for their ability to compute well-conditioned bases for the column space or row space of a matrix. For example, when $m<n$ and $k=m$ column-pivoted QR~\cite{GB1965} seeks to select $k$ well-conditioned columns of $A$ to span its column space. In this setting, $A_k=A$ and one focuses on~\cref{eq:GoodLeadingSV} together with so-called interpolative bounds (see~\cref{def:InterpolativeBoundsQR});~\cref{sec:QCapplication} discusses such an application.

There have been many proposed definitions of rank-revealers, and it is more common to have separate definitions for each type of factorization. For example, there is a definition of what it means for a partial QR factorization to be a rank-revealer~\cite{hong1992rank}. Likewise, there is a similar definition for the partial LU factorization~\cite{pan2000existence}. Instead, our definition of a rank-revealer is about the algorithm and not factorization, providing a more unifying viewpoint. An algorithm is a rank-revealer if it always provides good singular value estimates in the sense of~\cref{eq:GoodLeadingSV,eq:GoodTrailingSV}. A related, but informally stated, definition of a rank-revealer is given in~\cite{demmel1999computing}. In particular, our definition of a rank-revealer can even allow for randomized numerical linear algebra (rNLA) techniques if one only demands~\cref{eq:GoodLeadingSV,eq:GoodTrailingSV} to hold with high probability. There are many results in the rNLA literature that are close to~\cref{eq:GoodLeadingSV,eq:GoodTrailingSV} with some work focusing more on the conditioning of the computed low-rank approximation~\cite{meier2024fast,martinsson2020randomized} and others on bounding the norm of $A-A_k$~\cite{halko2011finding,martinsson2020randomized}.

In this paper, we focus on deterministic rank-revealers based on GE and QR as they appear as a subroutine in almost all the singular value estimators, including the ones in the rNLA literature~\cite{halko2011finding,martinsson2020randomized}. Rank-revealers in the presence of errors are also interesting~\cite{stewart1993determining}, though we focus on the situation where $A$ is known exactly. 

\subsection{Examples of rank-revealers} 
Rank-revealers fall into one of three categories: zero-sided, one-sided, and two-sided interpolative, where a ``side'' refers to the column or row space of $A_k$. The low-rank approximation $A_k$ is considered one-sided interpolative if $k$ columns of $A$ span the column space of $A_k$ or $k$ rows of $A$ span its row space. A two-sided interpolative low-rank approximation means that the column and row space of $A_k$ are spanned by $k$ columns and rows of $A$, respectively. In many applications, the columns and rows of $A$ have a physical meaning, and interpolative bases immediately lend themselves to a more straightforward interpretation. Three canonical rank-revealers are the truncated SVD (see~\cref{sec:truncatedSVD}), pivoted QR (see~\cref{sec:CPQR}), and GE with pivoting (see~\cref{sec:GE}). GE and QR compute partial LU and QR factorizations, respectively. The pivoting strategy used by GE and QR significantly affects the practical and theoretical properties of the rank-revealer (see~\cref{fig:Overview}). 

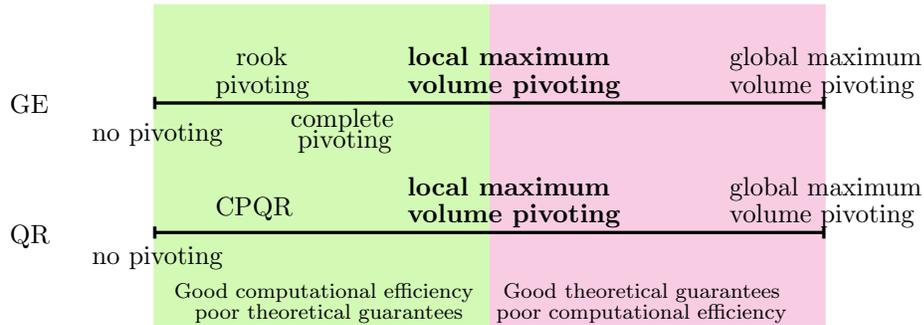
\begin{figure} 
\centering 
\begin{overpic}[width=.7\textwidth]{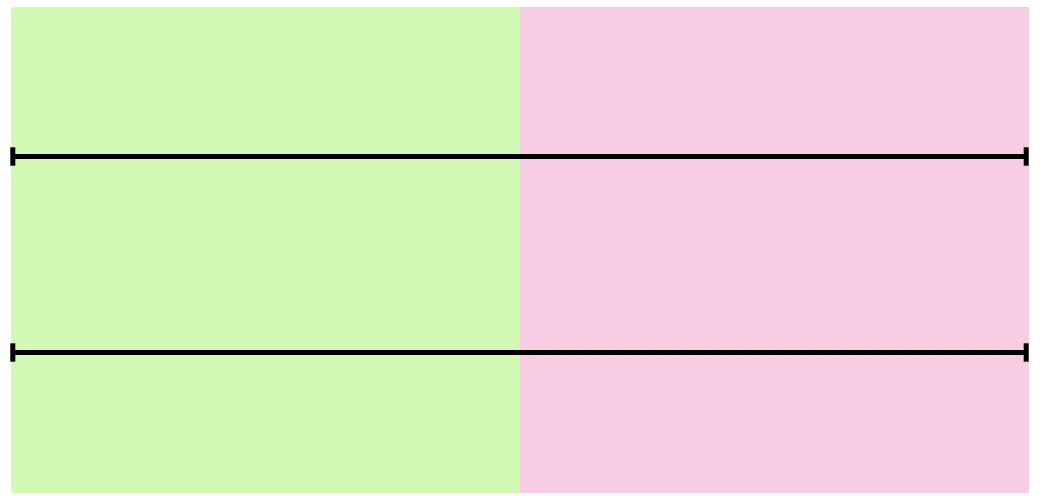}
\put(-20,32) {GE}
\put(-20,13) {QR}
\put(-8,28) {no pivoting}
\put(13,39) {rook}
\put(10,35) {pivoting}
\put(21,30) {complete}
\put(22,27) {pivoting}
\put(38,39) {{\bf local maximum}}
\put(38,35) {{\bf volume pivoting}}
\put(85,39) {global maximum}
\put(85,35) {volume pivoting}
\put(-8,10) {no pivoting}
\put(10,17) {CPQR}
\put(38,20) {{\bf local maximum}}
\put(38,16) {{\bf volume pivoting}}
\put(85,20) {global maximum}
\put(85,16) {volume pivoting}
\put(4,5) {{\footnotesize{Good computational efficiency}}}
\put(7,2) {{\footnotesize{poor theoretical guarantees}}}
\put(52,5) {{\footnotesize{Good theoretical guarantees}}}
\put(51,2) {{\footnotesize{poor computational efficiency}}}
\end{overpic} 
\caption{Different pivoting strategies in GE and QR can be more or less computationally efficient as well as better or worse singular value estimators. In this paper, we show that local maximum volume pivoting (see~\cref{sec:maxvolpivoting}) is a balance between computational efficiency and theoretical guarantees. Near-local maximum volume pivoting is a necessary and sufficient pivoting strategy for GE and QR to be rank-revealers (see~\cref{thm.sufficientLU,thm.sufficientQR,thm.necessaryLU,thm.necessaryQR2}). }
\label{fig:Overview}
\end{figure} 

\subsubsection{Zero-sided: Truncated SVD}\label{sec:truncatedSVD}
Truncating the SVD after $k$ terms is a rank-revealer that ensures that $\mu_{m,n,k}=1$ in~\cref{eq:GoodLeadingSV,eq:GoodTrailingSV}. Here, 
\[
A_k = \begin{bmatrix}| & \!\!\!\cdots\!\!\! & | \\[2pt]  u_1 & \!\!\!\cdots\!\!\! & u_k\\[2pt] | &\!\!\! \cdots\!\!\! & | \\\end{bmatrix} \!\! \begin{bmatrix}\sigma_1(A)\!\!\!\\[-4pt] & \!\!\!\!\!\ddots\!\! \!\!\!\\[-4pt] & & \!\!\!\sigma_k(A)\end{bmatrix}\!\! \begin{bmatrix}\raisebox{1mm}{\rule{.5cm}{0.5pt}} & v_1^\top & \raisebox{1mm}{\rule{.5cm}{0.5pt}} \\[-5pt] \vdots & \vdots & \vdots \\[-3pt] \raisebox{1mm}{\rule{.5cm}{0.5pt}} & v_k^\top & \raisebox{1mm}{\rule{.5cm}{0.5pt}} \end{bmatrix},
\]
where $u_1,\ldots,u_k\in\mathbb{R}^m$ and $v_1,\ldots,v_k\in\mathbb{R}^n$ are the first $k$ singular vectors and $\sigma_1(A)\geq \cdots \geq \sigma_k(A)\geq 0$ are the first $k$ singular values of $A$. Unfortunately, it is often prohibitively expensive to compute $A_k$ for large matrices as the computational cost of computing the full SVD is dominated in floating-point arithmetic by an $\mathcal{O}(mn\min{(m,n)})$ step. Suppose iterative methods are used (as is natural when $k\ll \min{(m,n)}$). The cost per iteration is typically $\cO(mn)$, and if convergence is fast, the overall computational cost can behave like $\cO(kmn);$ in practice; however, such approaches are often substantially slower than partial GE or QR. The truncated SVD is a zero-sided rank-revealer as the rank $\leq k$ approximant does not have a column or row space spanned by $k$ columns or rows of $A$. This can be an undesirable feature. For instance, if $A$ is a sparse matrix, then $A_k$ is usually dense.

\subsubsection{One-sided: Pivoted QR}\label{sec:CPQR}
A pivoted QR factorization works similarly to the standard QR factorization, except the columns are orthogonalized in a (potentially) different order, which is determined by a pivoting strategy. The column-pivoted QR (CPQR)~\cite{GB1965} is the most popular pivoted QR factorization. CPQR works in the same way as the standard QR factorization of a matrix, except for two differences: (i) It orthogonalizes the columns of $A$ in a different order by greedily selecting the next column to orthogonalize as the one with the largest remaining $2$-norm after projecting out those columns already selected and (ii) It stops after $k$ steps. After $k$ steps, CPQR decomposes $A$ into a partial QR factorization as 
\begin{equation}\label{eq.rankkQR}
AP = \begin{bmatrix}Q_1 & Q_2\end{bmatrix}\begin{bmatrix}R_{11} & R_{12} \\ 0 & R_{22} \end{bmatrix}, \qquad R_{12}\in\mathbb{R}^{k\times (n-k)}, \quad R_{22}\in\mathbb{R}^{(\min(m,n)-k)\times (n-k)},
\end{equation}
where $P \in \R^{n \times n}$ is a permutation matrix, $R_{11}\in\mathbb{R}^{k\times k}$ is upper-triangular, and $Q_1\in\mathbb{R}^{m\times k}$ and $Q_2\in\mathbb{R}^{m\times (\min(m,n)-k)}$ such that $Q_1$ has orthonormal columns. Despite using the notation, the factor $Q_2$ in~\cref{eq.rankkQR} is rarely explicitly computed in a partial QR factorization and does not need to have orthonormal columns. Likewise, $R_{22}$ does not need to be upper triangular. The permutation matrix is used here to order the columns so that $k$ steps of CPQR on $AP$ is equivalent to the first $k$ steps of an unpivoted QR. One can think of the first $k$ columns of $AP$, which correspond to $k$ columns from $A$, as the submatrix of $A$ selected by CPQR to use as the column space of $A_k$. A different strategy may select a different set of $k$ columns from $A$, changing the properties of the resulting partial QR factorization. 

For any $1\leq k\leq {\rm rank}(A)$, a pivoted QR factorization constructs a rank $k$ approximant given by $A_k = Q_1\begin{bmatrix}R_{11} & R_{12}\end{bmatrix}P^{-1}$ in only $\mathcal{O}(kmn)$ operations. For almost all matrices, CPQR is observed to provide good estimates of the singular values of $A$ with 
\[
\sigma_j(A_k) \approx \sigma_j(A), \quad 1\leq j\leq k, \qquad \sigma_{j}(A - A_k) \approx \sigma_{j+k}(A),  \quad 1\leq j\leq \min(m,n)-k.
\]
However, CPQR does not guarantee reasonable singular value estimates (see~\cref{sec:CPQRmetric}) and theoretically one only has $\mu_{m,n,k} = 2^k\sqrt{n-k}$ in ~\cref{eq:GoodLeadingSV,eq:GoodTrailingSV} (see~\cite[Thm.~7.2]{gu1996efficient}). The exponential growth of $\mu_{m,n,k}$ in $k$ is unacceptable theoretically. Nevertheless, CPQR usually provides excellent singular value estimates (see~\cref{fig:LUexperimentsmetric} (right)) and is widely used in practice.

Since a pivoted QR constructs a rank $k$ approximation that can be expressed as $A_k = Q_1R_{11}\begin{bmatrix}I_{k,k} & R_{11}^{-1}R_{12}\end{bmatrix}P^{-1}$, the singular values of $R_{11}$ and $A_k$ are certainly close when $\|R_{11}^{-1}R_{12}\|_{max}$ is not too large, where $\|\cdot\|_{\max}$ denotes the maximum absolute entry of a matrix and $I_{k,k}$ is the $k\times k$ identity matrix. It is computationally convenient if the singular values of $R_{11}$ closely match those of $A_k$ as in some applications, we can directly use $R_{11}$ to estimate the leading singular values of $A$. Therefore, a partial QR factorization with a small interpolative bound is often desired. 
\begin{definition} 
We say that the partial QR factorization in~\cref{eq.rankkQR} satisfies an interpolative bound with $\nu\geq 0$ if $\|R_{11}^{-1}R_{12}\|_{\max} \leq \nu$.
\label{def:InterpolativeBoundsQR} 
\end{definition}
Interpolative bounds are important because they ensure that the computed factors in the partial QR factorization can be used to calculate well-conditioned bases for a dominant column space of $A$ and nullspace~\cite{bischof1992signal}.

There are many other possible pivoting strategies (see~\cref{Tab:RRQR}), and one often tries to balance the computational efficiency of a QR pivoting strategy with its observed performance as a rank-revealer. A QR factorization with any pivoting strategy is a one-sided interpolative factorization because the columns of $Q_1$ are orthogonalized vectors coming from a subset of the columns of $A$. Hence, the column space of $A_k$ is always a linear combination of $k$ columns of $A$. 

\subsubsection{Two-sided: GE with pivoting}\label{sec:GE}
GE with a pivoting strategy selects the order of the columns and rows to eliminate based on a pivoting strategy. A common strategy used in the rank-revealing literature is GE with complete pivoting (GECP). This works in the same way as GE without pivoting, except the columns/rows are eliminated in a different order. The next pivot is selected as an entry with the largest absolute magnitude, and the column and row containing the pivot entry are eliminated. In this paper, GE is always terminated after $k$ steps to decompose $A$ into the following partial LU factorization: 
\begin{equation}\label{eq.rankkLU}
P_1AP_2 \!=\! \begin{bmatrix}A_{11} \!\! & \!\!\!A_{12} \\ A_{21}\!\! & \!\!\!A_{22}\end{bmatrix} \!=\! \begin{bmatrix}I_{k,k} \!\! & \!\!0 \\ A_{21}A_{11}^{-1}\!\! & \!\!I_{m-k,n-k}\end{bmatrix} \!\!\begin{bmatrix}A_{11}\!\! & \!\!\!A_{12} \\ 0\!\! & \!\!\!S(A_{11})\end{bmatrix}\!,\quad \!\!\!\!S(A_{11}) = \!A_{22} - A_{21}A_{11}^{-1}\!A_{12},
\end{equation}
where $I_{p,q}$ is the $p\times q$ matrix with ones on the diagonal and $P_1\in \R^{m \times m}$ and $P_2 \in \R^{n \times n}$ are permutation matrices.  Here, $A_{11}\in\mathbb{R}^{k\times k}$, $A_{12}\in\mathbb{R}^{k\times (n-k)}$, $A_{21}\in\mathbb{R}^{(m-k)\times k}$, and $A_{22}\in\mathbb{R}^{(m-k)\times (n-k)}$ are submatrices of $P_1AP_2$. 

If~\cref{eq.rankkLU} is the partial factorization formed by GECP, then the permutation matrices $P_1$ and $P_2$ ensure that the first $k$ steps of GECP do no pivoting. In this way, GECP on $P_1AP_2$ is equivalent to unpivoted GE when terminated after $k$ steps. We call $A_{11}$ the $k\times k$ GECP pivot of $A$. 

In the pseudoskeleton literature, another popular choice of a rank-revealer is GE with global maximum volume pivoting~\cite{goreinov2010find,goreinov1997theory}. This method works similarly to GECP, except (i) It selects a $k\times k$ pivot all at once instead of building it up element by element, (ii) It selects the $k\times k$ pivot that has the largest volume (see~\cref{def:volume}), i.e., the largest possible determinant in absolute value among all $k\times k$ submatrices. For GE with global maximum volume pivoting, the submatrix $A_{11}$ in~\cref{eq.rankkLU} is a $k\times k$ submatrix with the largest volume. 

Unfortunately, it is costly to compute the global maximum volume pivot as it takes $\mathcal{O}(m^kn^k)$ operations. The pivoting strategy takes exponential time because it must essentially compute the volume of every $k\times k$ submatrix of $A$~\cite{CIVRIL20094801}. Despite this, GE with global maximum volume pivoting is interesting from a theoretical perspective as $\mu_{m,n,k}=1+5k\sqrt{mn}$ (see~\Cref{thm.sufficientLU}). While CPQR is practical but not theoretically rigorous, GE with global maximum volume pivoting is at the other extreme (see~\cref{fig:Overview}). 

For any $1\leq k\leq {\rm rank}(A)$, GE with pivoting constructs a rank $k$ approximant $A_k$ given by 
\begin{equation}\label{eq.approxLU}
A_k = P_1^{-1}\begin{bmatrix}A_{11} \!\! & \!\!\!A_{12} \\ A_{21}\!\! & \!\!\!A_{21}A_{11}^{-1}A_{12}\end{bmatrix}P_2^{-1} = P_1^{-1}\begin{bmatrix}I_{k,k} \\A_{21}A_{11}^{-1}\end{bmatrix} A_{11} \begin{bmatrix} I_{k,k} & A_{11}^{-1}A_{12} \end{bmatrix}P_2^{-1},
\end{equation}
where the last statement is a low-rank representation of $A_k$. All partial LU factorizations of the form in~\cref{eq.approxLU} are two-sided interpolative factorizations because the columns and rows of $A_k$ come from linear combinations of $k$ columns and rows of $A$.

It is insightful to write $A_k$ in low-rank form in~\cref{eq.approxLU} as one can see that when the entries of $A_{21}A_{11}^{-1}$ and $A_{11}^{-1}A_{12}$ are not too large, the singular values of $A_k$ and $A_{11}$ are close. Therefore, analogously to QR with pivoting, we often want to compute a partial LU factorization with small interpolative bounds. 
\begin{definition} 
We say that the partial LU factorization in~\cref{eq.rankkLU} satisfies interpolative bounds with $\nu\geq 0$ if $\|A_{21}A_{11}^{-1}\|_{\max} \leq \nu$ and $\|A_{11}^{-1}A_{12}\|_{\max} \leq \nu$.
\label{def:InterpolativeBoundsGE} 
\end{definition}

Like in QR, interpolative bounds are important for GE because they ensure that the computed factors in the partial LU factorization can be used to calculate well-conditioned bases for subspaces associated with $A$. 

\subsection{Our contributions} 
We unify the landscape of rank-revealing algorithms by considering GE and QR together. Our contributions demonstrate that local maximum volume pivoting is the central concept for rank-revealers based on GE and QR. 

\begin{enumerate}[leftmargin=*]
\item {\bf Local maximum volume pivoting is sufficient and practical.}
We show that GE and QR with local maximum volume pivoting constructs a rank $k$ approximation $A_k$ that satisfies~\cref{eq:GoodLeadingSV,eq:GoodTrailingSV} with $\mu_{m,n,k} = 1 + 5k\sqrt{mn}$ and $\mu_{m,n,k} = \sqrt{1+5kn}$, respectively (see~\cref{thm.sufficientLU,thm.sufficientQR}). These values for $\mu_{m,n,k}$ are definitive and cannot be asymptotically improved. Moreover, local maximum volume pivoting also delivers partial LU and QR factorizations satisfying interpolative bounds with $\nu=1$. Interestingly, \emph{any} local maximum volume pivot has this property---irrespective of its distance to the global maximum volume subset. In essence, all local minimizers on the graph of subsets with edges representing single swaps (see~\cref{sec:volumegraph}) are good. Finally, local maximum volume pivots can often be found computationally efficiently (see~\cref{alg:localMaxVol}), making these pivoting strategies practical. There is a GitHub repository accompanying this manuscript~\cite{Github}, which contains MATLAB scripts of practical near-local maximum volume GE and QR rank-revealer. 

\item {\bf Near-local maximum volume pivoting is theoretically necessary.}
We prove that any partial LU or QR factorization of $A$ that provides good estimates for the singular values of $A$ and satisfies interpolative bounds finds a near-local maximum volume submatrix (see~\cref{thm.necessaryLU,thm.necessaryQR2}). Therefore, near-local maximum volume pivoting is crucial for GE and QR to be rank-revealers. 

\item {\bf Local maximum volume as a metric.}
Since near-local maximum volume pivoting is essential for rank-revealers based on GE and QR, we design an efficient metric to assess the quality of any pivoting strategy (see~\cref{sec:assessment}). This can be used in a practical setting to see if a particular pivot is reasonable.  In the GitHub repository~\cite{Github}, we give MATLAB scripts for efficiently computing the metric. 
\end{enumerate}

\subsection{Impact}\label{sec:Impact}

Rank-revealers have a substantial impact across various fields. In no particular order, we give a taste of the vast range of applications here: 

\begin{itemize}[leftmargin=*]
\item[] \textbf{Interpretable low-rank approximation:} They are essential for constructing low-rank models that are interpretable\cite{mahoney2009cur,xia2024making,cheng2005compression,sorensen2016deim}, aiding in the understanding of underlying data structures in bioinformatics, social network analysis, and other fields. 
\item[] \textbf{Fast solvers for integral equation formulations:} They play a key role in fast and accurate computational schemes for working with matrices that arise in integral equations formulations of PDEs within applications such as electromagnetics and fluid dynamics~\cite{ho2012recursive,xia2010fast,martinsson2005fast,martinsson2019fast}.
\item[]\textbf{Fast algorithms for kernel matrices:} They facilitate the efficient manipulation of large kernel matrices and therefore play a key role in machine learning models such as Gaussian processes and kernelized support vector machines~\cite{ambikasaran2015fast,minden2017fast,drineas2005nystrom}. 
\item[] \textbf{Function approximation:} A continuous analogue of GE on function is used for low-rank function approximation~\cite{townsend2013extension} (see~\cref{sec:FunctionApproximation}). Related ideas go by the names adaptive cross approximation~\cite{bebendorf2000approximation}, pseudoskeleton approximation~\cite{goreinov1997theory}, and Geddes--Newton series~\cite{carvajal2005hybrid}. 
\item[] \textbf{Computing localized orbitals in computational quantum chemistry:} They enable the robust and efficient computation of localized orbitals, which play a key role in both the analysis of atomic systems and construction of efficient computational schemes~\cite{damle2015compressed,damle2018disentanglement} (see~\cref{sec:QCapplication}). 
\item[] \textbf{Separable non-negative matrix factorization:} They form the basis of robust and accurate factorization methods in image processing and text mining, where extracting meaningful components is key~\cite{gillis2013fast}.
\item[] \textbf{Spectral clustering:} They can be used as simple and efficient methods for spectral clustering, leading to better data segmentation and analysis outcomes~\cite{damle2019simple}. 
\item[] \textbf{Model order reduction:} They are pivotal for simplifying complex dynamical systems through efficient model order reduction techniques such as QDEIM~\cite{drmac2016new} (see~\cref{sec:ROMapplication}).
\item[] \textbf{Compressing deep learning models:} They are a key part of algorithms to effectively compress deep learning models that do not require ``fine-tuning'' to retain accuracy---enabling the use of such models in resource-constrained environments~\cite{chee2022model}.
\item[] \textbf{Archetype analysis:} They aid in archetype analysis, helping identify prototypical profiles in various datasets---a vital task for exploratory data analysis~\cite{javadi2019nonnegative}. 
\item[] \textbf{Machine learning and volume sampling:} They have intimate connections to machine learning algorithms that rely on volume sampling for more efficient handling of large data sets~\cite{derezinski2020improved,deshpande2010efficient,derezinski2018leveraged}. 
\item[] \textbf{Determinantal point processes (DPP) sampling:} They can help explain the effectiveness of DPP sampling strategies used in recommender systems, spatial statistics, and other applications~\cite{derezinski2021determinantal,anari2022sampling}.
\item[] \textbf{Computational algebraic geometry:} The column subset selection problem appears in the construction of the M\"{o}ller--Stetter matrix using the Macaulay's resultant~\cite{telen2018solving}. 
\end{itemize}

\section{Local maximum volume pivoting}\label{sec:maxvolpivoting}
For GE or QR to be rank-revealers, we will show that it is important that they employ a pivoting strategy that selects a submatrix with a large volume, which is defined as follows:\footnote{Confusingly, the volume for tall-skinny matrices has sometimes been defined as ${\rm det}(A^\top A)$~\cite{mikhalev2018rectangular}.}
\begin{definition}[Volume]
The volume of any matrix $A\in\mathbb{R}^{m\times n}$ is given by the product of its singular values. That is, 
\[
	{\rm vol}(A) = \prod_{j=1}^{\min(m,n)} \sigma_j(A).
\]
\label{def:volume}
\end{definition} 
When $A$ is a square matrix, then ${\rm vol}(A) = \abs{{\rm det}(A)}$. Intuitively, the columns or rows are well-conditioned when a matrix has a large volume. Similarly, if the volume of a matrix is zero, then the columns or rows are linearly dependent. 

GE with global maximum pivoting selects the pivot as the $k\times k$ submatrix with the largest volume. Finding the global maximum volume pivot is typically too computationally expensive to be helpful in most practical applications~\cite{CIVRIL20094801}. Instead, one can preserve most of the rank-revealing properties of GE (see~\cref{sec:LUsufficient}) and QR (see~\cref{sec:sufficientQRsection}) by using a strategy that is called local maximum volume pivoting, which is an idea from Hong and Pan's work in the early 1990s~\cite{hong1992rank,pan2000existence}. 

\begin{definition}[Local maximum volume]\label{dfn:maxvolgeneral}
Let $A \in \R^{m \times n}$. We say that a submatrix $B \in \R^{p \times q}$ of $A$ is a local maximum volume submatrix if ${\rm vol}(B)>0$ and
\[
	{\rm vol}(B) \geq {\rm vol}(\hat{B})
\]
for all $p \times q$ submatrices $\hat{B}$ of $A$ that differ from $B$ in at most one row and at most one column.\footnote{We insist on ${\rm vol}(B)>0$ to exclude singular submatrices.}
\end{definition}

In 1996, Gu and Eisenstat used local maximum volume pivoting for QR~\cite{gu1996efficient} and, in 2003, Miranian and Gu used it as a pivoting strategy for GE~\cite{miranian2003strong}. Local maximum volume submatrices also go by the name  of locally dominant submatrices in the pseudoskeleton literature~\cite{goreinov2010find,mikhalev2018rectangular}. A similar but totally distinct concept is locally optimal submatrices in random matrix theory~\cite{bhamidi2017energy}. The concept of volume also plays a key role in so-called volume sampling methods, which are randomized methods~\cite{deshpande2006matrix,deshpande2010efficient,derezinski2018leveraged}. 

Of course, any global maximum volume submatrix of a nonzero matrix $A$ is also a local maximum volume submatrix, but not vice versa (see~\cref{example.simplemat}). To check that $B$ is a local maximum volume $p\times q$ submatrix of $A$, one must check that ${\rm vol}(B) \geq {\rm vol}(\hat{B})$ for all possible $\hat{B}$. Since $\hat{B}$ is a submatrix of $A$ and is equal to $B$ except in at most one column and row, we can explicitly count how many $\hat{B}$'s there are. There are $p$ possible rows to remove from $B$ and $m-p$ potential rows to replace it, and after this choice, there are $q$ possible columns to remove and $n-q$ replacements. Allowing for the possibility of not swapping out a row or column but ensuring $B\neq\hat{B}$, we have $(p(m-p)+1)(q(n-q)+1)-1$ possible $\hat{B}$'s. So, checking if a submatrix has a local maximum volume is much faster than verifying that it is globally optimal. Moreover, checking if ${\rm vol}(B)\geq {\rm vol}(\hat{B})$ can be done more efficiently (and stably) than computing the product of the singular values of $B$ and $\hat{B}$. Instead, one can use the fact that $\hat{B}$ is at most a rank two update of $B$ (see~\cref{eq.volinterpreiterate,eq:UpdateQR}).

\subsection{The volume submatrix graph}\label{sec:volumegraph}
It is helpful to consider the graph whose nodes are all $p \times q$ submatrices of $A$ and where an edge connects two nodes if they differ by at most one row and at most one column. Each node of this graph has degree $(p(m-p)+1)(q(n-q)+1)-1$ and can be assigned the value of the volume of its associated submatrix. We call this graph the {\em volume submatrix graph}. The nodes and edges of a volume submatrix graph do not depend on $A$, though the node values do.  On this graph, the node with the largest value corresponds to a global maximum volume pivot. Any local maximum submatrix corresponds to a node whose neighbors do not have larger volumes.

\begin{example}\label{example.simplemat}
Consider the following $4\times 4$ block diagonal matrix and its 36-node graph of $2\times 2$ submatrices: 

\begin{minipage}{.31\textwidth} 
\[
A =\begin{bmatrix}
		1 & 3 & 0 & 0 \\
		3 & 1 & 0 & 0  \\
		0 & 0 & \sqrt{3} & 2 \\
		0 & 0 & 2 & -\sqrt{3}
	\end{bmatrix}
\] 
\end{minipage} 
\begin{minipage}{.63\textwidth} 
\centering
\begin{overpic}[width=.6\textwidth,trim=0 0 18 0,clip]{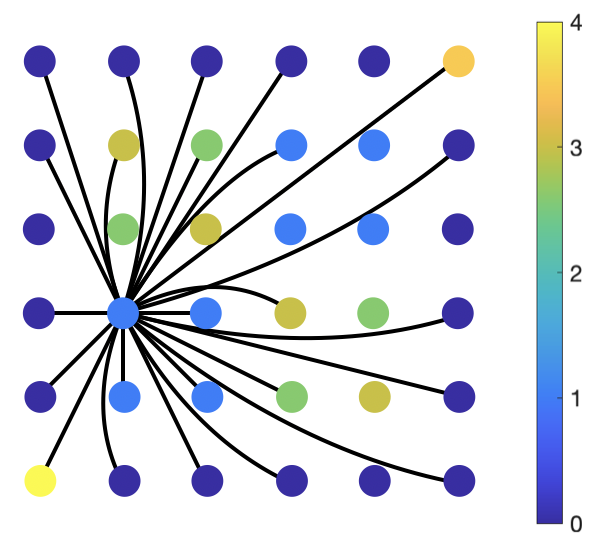}
\put(0,1) {\footnotesize{$(1,\!2)$}}
\put(15,1) {\footnotesize{$(1,\!3)$}}
\put(30,1) {\footnotesize{$(1,\!4)$}}
\put(46,1) {\footnotesize{$(2,\!3)$}}
\put(60,1) {\footnotesize{$(2,\!4)$}}
\put(76,1) {\footnotesize{$(3,\!4)$}}
\put(-12,9) {\footnotesize{$(1,\!2)$}}
\put(-12,24) {\footnotesize{$(1,\!3)$}}
\put(-12,38) {\footnotesize{$(1,\!4)$}}
\put(-12,53) {\footnotesize{$(2,\!3)$}}
\put(-12,68) {\footnotesize{$(2,\!4)$}}
\put(-12,83) {\footnotesize{$(3,\!4)$}}
\put(105,80) {\rotatebox{270}{\text{Volume of submatrix}}}
\put(100,3) {$0$}
\put(100,46.5) {$4$}
\put(100,25) {$2$}
\put(100,68) {$6$}
\put(100,90) {$8$}
\end{overpic}
\end{minipage} 

Here, we have only shown the edges in the graph connecting to one of the nodes, which has $(2^2+1)^2-1 = 24$ edges. If a node is at position labeled $(1,4)\times (1,3)$, then it corresponds to the $2\times 2$ submatrix $A([1,4],[1,3])$. The color of the node indicates the volume of that submatrix. The matrix $A$ has two local maximum volume $2\times 2$ submatrices, which are $A([1,2],[1,2])$ and $A([3,4],[3,4])$.
\end{example} 

\subsection{Computing a local maximum volume submatrix} 
One of the most delightful properties of local maximum submatrices is that there is a simple algorithm to find one of them. Start at any node with a nonzero node value in the volume submatrix graph and look at its neighbors. As soon as you find a neighbor with a larger volume, move to that node and start examining the neighbors of that new node. If, at any point, you reach a node for which all its neighbors have the same or smaller volume, then you have found a local maximum volume submatrix (see~\cref{alg:localMaxVol}). 

\Cref{alg:localMaxVol} forms a path on the volume submatrix graph of nodes with increasing associated volumes. The computational cost of~\cref{alg:localMaxVol} depends on how long the path is. In our numerical experience, the paths taken by~\cref{alg:localMaxVol} are almost always short on random matrices.\footnote{As we will see in~\cref{sec:applications} they are also often short for application driven problems.} In~\cref{fig:pathlengths} (left), we generate a random matrix and investigate the typical path lengths of this algorithm.\footnote{The matrix is generated in MATLAB R2023a using the command \texttt{rng(128); A = randn(11,11)}. We use the seed 128 so that the paths look aesthetically pleasing.}

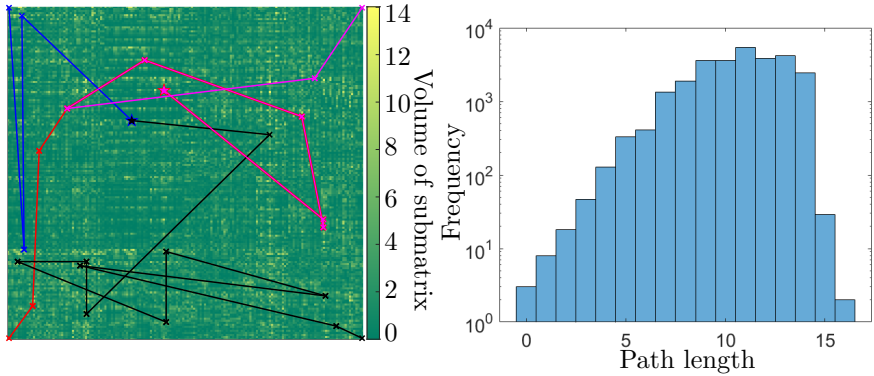
\begin{figure}
\centering
\begin{minipage}{.49\textwidth}
\begin{overpic}[width=.9\textwidth,trim=0 0 80 0,clip]{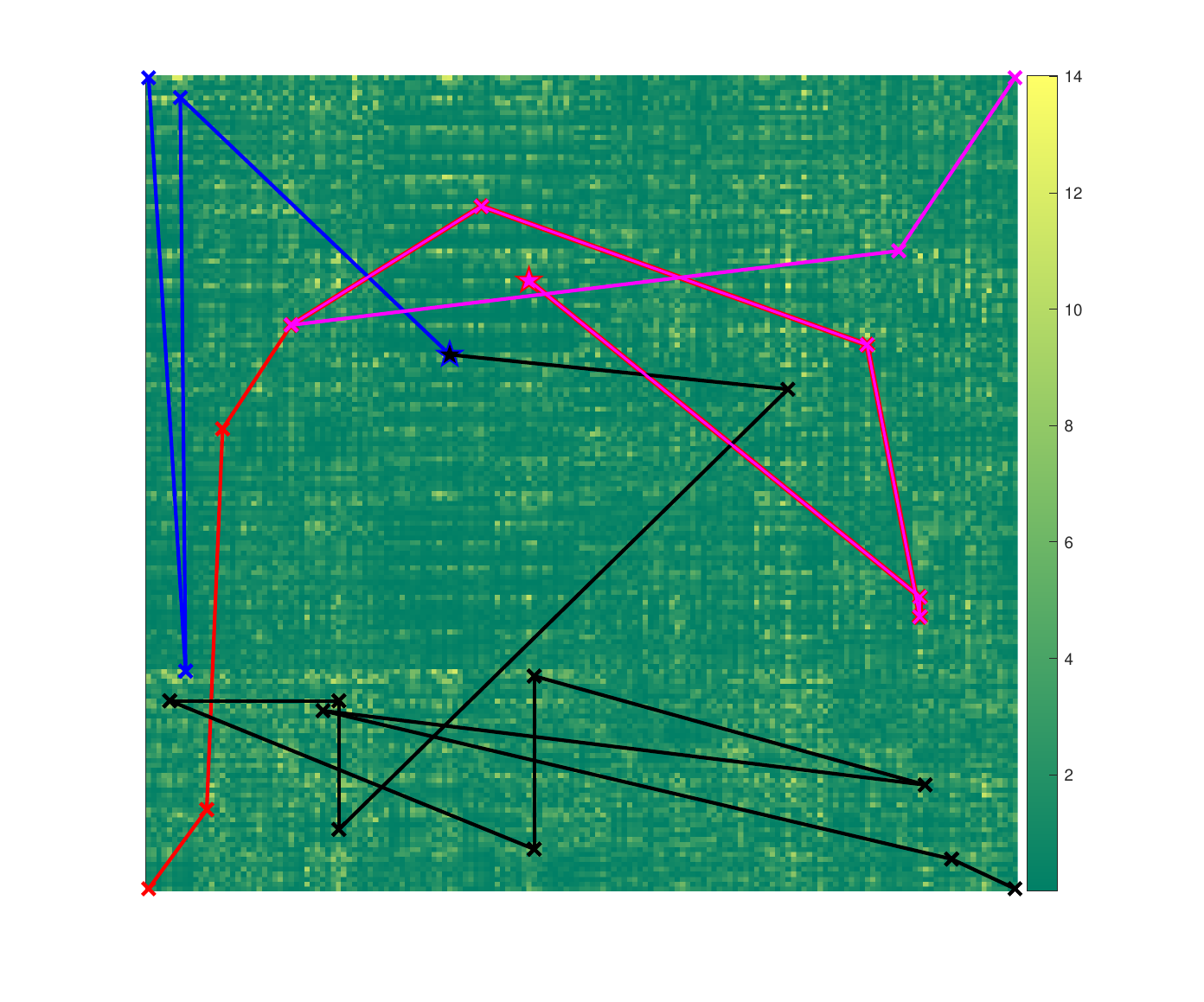}
\put(108,73) {\rotatebox{270}{\text{Volume of submatrix}}}
\put(101,10) {0}
\put(101,20) {2}
\put(101,31) {4}
\put(101,41) {6}
\put(101,53) {8}
\put(101,63) {10}
\put(101,74) {12}
\put(101,84) {14}
\end{overpic} 
\end{minipage} 
\begin{minipage}{.49\textwidth}
\begin{overpic}[width=\textwidth]{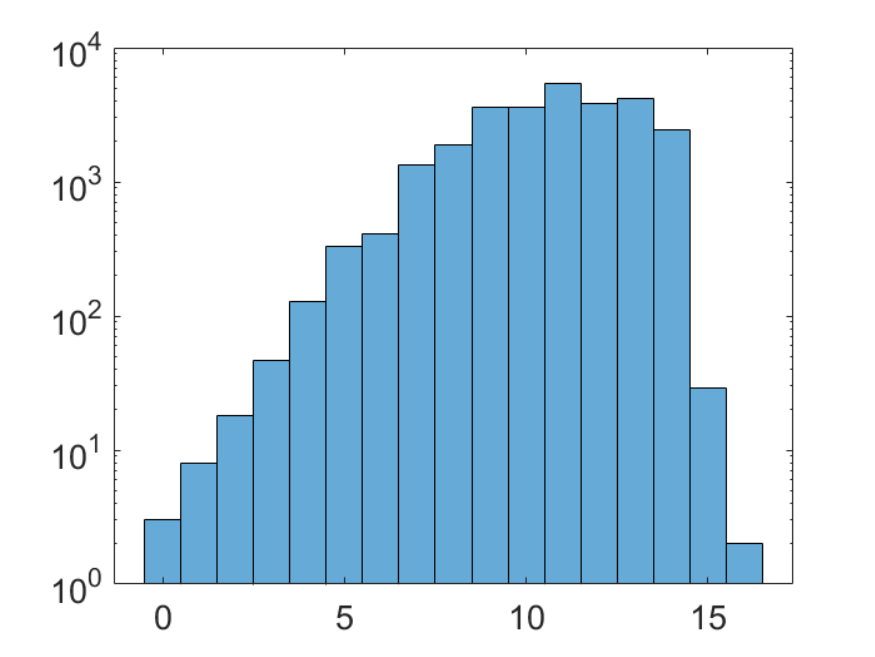}
\put(38,-1.5) {Path length}
\put(1,25) {\rotatebox{90}{Frequency}}
\end{overpic} 
\end{minipage} 
\caption{Left: Four representative paths taken by~\cref{alg:localMaxVol} on the volume submatrix graph for an $11\times 11$ random Gaussian matrix with $3\times 3$ submatrices. There are $27,225$ nodes in the volume submatrix graph. The graph is so large that the individual nodes have merged; instead, we see a color map of the node values. Starting at these four corner submatrices,~\cref{alg:localMaxVol} finds local maximum volume submatrices with path lengths of 3 (blue), 7 (magenta), 8 (red), and 10 (black). Two of the paths coalesce (red and magenta) and then follow each other after that. Right: Histogram of the distribution of path lengths to find a local maximum volume submatrix using~\cref{alg:localMaxVol}, starting at all $27,225$ nodes. Despite there only being three local maximum volume submatrices, the maximum path length to find one is 16.} 
\label{fig:pathlengths}
\end{figure} 

\begin{algorithm}
\caption{Finding a local maximum volume $p\times q$ submatrix of a nonzero matrix.}
\label{alg:localMaxVol}
\begin{algorithmic}[1]
\State Start at any $p\times q$ submatrix $B$ of $A$ with ${\rm vol}(B)>0$ (see~\cref{def:volume}).  
\For{each neighbor $\hat{B}$ of $B$ (see~\cref{sec:volumegraph})}
\State Compute ${\rm vol}(\hat{B})/{\rm vol}(B)$ (see~\cref{eq.volinterpreiterate,eq:UpdateQR} for efficient formulas).
\If{${\rm vol}(\hat{B})/{\rm vol}(B)>1$}
\State Replace $B$ with $\hat{B}$ and go back to step 2. 
\EndIf
\EndFor
\State The algorithm terminates with a local maximum volume submatrix $B$. 
\end{algorithmic}
\end{algorithm}

This paper studies pivoted GE and QR with pivoting strategies that are based on selecting a submatrix of local maximum volume. For GE, the submatrices are of size $k\times k$ while, for QR, the submatrices are of size $m\times k$. Hence, for QR, one seeks $k$ columns with a volume that cannot be increased by replacing a single column. One can always start~\cref{alg:localMaxVol} at a greedily-selected initial submatrix. For $k\times k$ submatrices, one can select a submatrix found by GECP and for $m\times k$ submatrices, the one selected by CPQR. While these initial submatrices do not improve the theoretical performance of~\cref{alg:localMaxVol} as pathlengths can still be very long in principle, they do tend to reduce the pathlengths in practice.  We show that the local maximum volume pivoting strategy ensures GE and QR construct rank-revealing factorizations (see~\cref{thm.sufficientLU,thm.sufficientQR}). Later, we prove that if a pivot strategy does not guarantee at least a near-local maximum volume submatrix, then GE or QR will fail to either be a rank-revealer or will compute a partial factorization that does satisfy interpolative bounds (see~\cref{thm.necessaryLU,thm.necessaryQR2}). 

\section{GE with local maximum volume pivoting}
\label{sec:LUsufficient}
GE with local maximum volume pivoting works similarly to GE with global maximum volume pivoting (see~\cref{sec:GE}), except one uses a $k\times k$ local maximum volume submatrix as a pivot. In particular, this version of GE does not pivot on single entries. Using the same notation as in~\cref{eq.rankkLU} and~\cref{eq.approxLU}, we find that 
\begin{equation} 
P_1(A - A_k)P_2 = \begin{bmatrix} 0 & 0 \\ 0 & S(A_{11})\end{bmatrix}, 
\label{eq:schurcomplement} 
\end{equation}
where $S(A_{11})$ is a Schur complement (see~\cref{eq.rankkLU}) and $A_{11}$ is a local maximum volume $k\times k$ submatrix of $A$ (see~\cref{dfn:maxvolgeneral}). Since $P_1$ and $P_2$ are permutation matrices,~\cref{eq:schurcomplement} tells us that $\sigma_j(A-A_k) = \sigma_j(S(A_{11}))$ for $1\leq j\leq \min(m,n)-k$. There has been some research on rank-revealers based on GE with different pivoting strategies (see~\cref{Tab:RRLU}). In this section, we show that GE with local maximum volume pivoting is a rank-revealer. 

\begin{table}
\caption{Summary of deterministic GE-based approaches for estimating singular values. In~\cite{miranian2003strong}, an algorithm is not explicitly presented though their proof techniques suggest using GE with near-local maximum volume pivoting.}
\setlength{\tabcolsep}{5pt}
\begin{center}
\begin{tabular}{llrll}
	\hline\\[-2ex]
	Method & $\mu_{m,n,k}$ in~\cref{eq:GoodLeadingSV}\&\cref{eq:GoodTrailingSV} & Reference  \\ [0.5ex] 
	\hline\\[-2ex]
	Complete & $\mathcal{O}(4^{k}(k+\rho_k)\sqrt{mn})$ & (see~\cref{lem:GECPvolume}) &  \\
	Alg.~1 & $\mathcal{O}(k\sqrt{mn})$ & Schork+Gondzio, 2020 & \cite{schork2020rank} \\
    	Alg.~3 & $\mathcal{O}(k\sqrt{mn})$ & Pan, 2000 & \cite{pan2000existence} \\
	\hline
\end{tabular}
\label{Tab:RRLU}
\end{center}
\end{table}

A nice feature of GE with local maximum pivoting is that it constructs a partial LU factorization satisfying interpolative bounds with $\nu=1$. To see this, note that Cramer's rule shows that~\cite[Prop.~1]{miranian2003strong}
\begin{equation} 
\left|(A_{11}^{-1}A_{12})_{i,j}\right| = \left|\frac{{\rm det}(\hat{A}_{11})}{{\rm det}(A_{11})}\right| =\left|\frac{{\rm vol}(\hat{A}_{11})}{{\rm vol}(A_{11})}\right|  \leq 1, \qquad 1\leq i\leq k,\quad 1\leq j\leq n-k,
\label{eq:InterpolativeBounds}
\end{equation} 
where $\hat{A}_{11}$ is formed by replacing the $i$th column of $A_{11}$ by the $j$th column of $A_{12}$. Here, the bound of $1$ comes from the fact that $A_{11}$ is a local maximum volume submatrix of $A$. The same argument applies to the entries of $A_{21}A_{11}^{-1}$ to show that they are also all bounded by 1 in absolute value. Therefore, the partial LU factorization constructed by GE with local maximum volume pivoting satisfies interpolative bounds with $\nu = 1$. 

\subsection{GE with local maximum volume pivoting is a rank-revealer}\label{sec:GEisRankRevealer}
We are ready to prove that GE with local maximum volume pivoting is a rank-revealer. The result is related to a statement proved by Miranian and Gu on partial LU factorization~\cite[Thm.~2]{miranian2003strong}, except we consider the low-rank approximation $A_k$ rather than the $k\times k$ pivot matrix $A_{11}$. 

\begin{theorem}\label{thm.sufficientLU}
GE with local maximum volume pivoting is a rank-revealer with $\mu_{m,n,k} = 1+5k\sqrt{mn}$ (see~\cref{eq:GoodLeadingSV} and~\cref{eq:GoodTrailingSV}) and computes a partial LU factorization satisfying interpolative bounds with $\nu\leq 1$ (see~\cref{def:InterpolativeBoundsGE}). 
\end{theorem}
\begin{proof}
If GE with local maximum volume pivoting on $A$ is possible, then we know that ${\rm rank}(A)\geq k$; otherwise, all $k\times k$ submatrices of $A$ are singular, and there are no local maximum volume submatrices. From~\cite[Thm.~2]{miranian2003strong}, we have 
\begin{equation} 
\begin{aligned}
\sigma_j(A_{11}) \leq \sigma_j(A) &\leq \left(2k\sqrt{(m\!-\!k)(n\!-\!k)} \!+\! \sqrt{(1\!+\!k(n\!-\!k))(1\!+\!k(m\!-\!k))}\right)\sigma_j(A_{11}) \\
&\leq \left(1+5k\sqrt{mn}\right)\sigma_j(A_{11}), \quad 1\leq j\leq k,
\end{aligned}
\label{eq:FirstInequality}
\end{equation} 
where the lower bound comes from the fact that $A_{11}$ is a submatrix of $A$~\cite{thompson1972principal}.
Moreover, we have
 \begin{equation} 
 \begin{aligned}
\sigma_{k+j}(A) \leq\! \sigma_j(S(A_{11})) \!& \leq\! \left(2k\sqrt{(m\!-\!k)(n\!-\!k)} \!+\!\! \sqrt{(1\!+\!k(n\!-\!k))(1\!+\!k(m\!-\!k))}\right)\! \sigma_{k+j}(A)\\
&\leq \left(1+5k\sqrt{mn}\right)\sigma_{k+j}(A), \quad 1\leq j\leq \min(m,n)\!-\!k,
\end{aligned}
\label{eq:eq2} 
\end{equation} 
where the upper bound comes from~\cite[Thm.~2]{miranian2003strong} and the lower bound is a consequence of $\sigma_j(S(A_{11})) = \sigma_j(A-A_k)$ and the Eckart--Young Theorem~\cite{eckart1936approximation}.  Since $\sigma_j(S(A_{11})) = \sigma_j(A-A_k)$, the two inequalities in~\cref{eq:eq2} prove that GE with local maximum volume pivoting delivers the inequalities in~\cref{eq:GoodTrailingSV}. 

To complete the proof we must also show that the two inequalities in~\cref{eq:GoodLeadingSV} also hold. For the upper bound in~\cref{eq:GoodLeadingSV}, the low-rank form of $A_k$ in~\cref{eq.approxLU} shows that
\begin{equation} 
\sigma_j(A_k) \leq \left(1+\|A_{21} A_{11}^{-1}\|_2\right) \left(1+\|A_{11}^{-1} A_{12}\|_2\right)\sigma_j(A_{11}),\quad 1\leq j\leq k.
\label{eq:betweenA11Ak}
\end{equation} 
Moreover, we have $\|A_{11}^{-1} A_{12}\|_2\leq\sqrt{k(n-k)}\|A_{11}^{-1} A_{12}\|_{\max}\leq  \sqrt{k(n-k)}$, where the last inequality comes from~\cref{eq:InterpolativeBounds} and $\|\cdot\|_{\max}$ is the maximum absolute entry norm. Similarly, we have $\|A_{11}^{-1} A_{12}\|_2\leq  \sqrt{k(m-k)}$. Thus, for $1\leq j\leq k$, we find that $\sigma_j(A_k) \leq (1+\sqrt{k(n-k)})(1+\sqrt{k(m-k)})\sigma_j(A)$ as $\sigma_{j}(A_{11})\leq \sigma_j(A)$ from~\cref{eq:FirstInequality}. The upper bound follows as $(1+\sqrt{k(n-k)})(1+\sqrt{k(m-k)})\leq 1+ 3k\sqrt{mn}$. The theorem statement is proved as the upper bound in~\cref{eq:GoodLeadingSV} appears in~\cref{eq:FirstInequality} after noting that $A_{11}$ is a submatrix of $A_k$ so $\sigma_j(A_{11})\leq \sigma_j(A_k)$ for $1\leq j\leq k$. 
\end{proof}

\Cref{thm.sufficientLU} shows that GE with local maximum volume pivoting can estimate the singular values of $A$ to within a factor of $1+5k\sqrt{mn}$. One can use its constructed rank $\leq k$ approximation $A_k$ to estimate the first $k$ singular values of $A$. Since $A_k$ is of rank $k$, one can compute these singular values in $\mathcal{O}(k^2(m+n))$ operations~\cite[Sec.~1.1.4]{bebendorf2008hierarchical}. Moreover, we know that $A_k$ is a near-optimal low-rank approximation to $A$ as $\|A-A_k\|_2\leq (1+5k\sqrt{mn})\sigma_{k+1}(A)$.  Since the singular values of $A-A_k$ are also good estimates for the trailing singular values of $A$, if one wants to compute more than the first $k$ singular values of $A$, then there is no need to start GE afresh with a larger value of $k$. Instead, one can run GE with local maximum volume pivoting on the residual $A-A_k$. 

One might wonder if a significantly stronger statement than~\cref{thm.sufficientLU} is possible for GE with local maximum volume pivoting. The following example illustrates that this is not the case, and one could only improve the value of $\mu_{m,n,k}$ by at most constant factors. 

\begin{example} 
For any $m,n\geq 2$ and $2\leq k\leq \min(m,n)$ consider the following $m\times n$ matrix whose entries depend on $k$: 
\[
A =\left[\begin{array}{cccc|ccc}
        k+1 & -1 & \cdots & -1 & -1 & \cdots & -1 \\
        -1 & k+1 & \cdots & -1 & -1 & \cdots & -1 \\
        \vdots & \vdots & \ddots & \vdots & \vdots & \ddots & \vdots \\
       -1 & -1 & \cdots & k+1 & -1 & \cdots & -1 \\
        \hline
        -1 & -1 & \cdots & -1 & k+1 & \cdots & k+1 \\
        \vdots & \vdots &\ddots & \vdots & \vdots & \ddots & \vdots \\
        -1  & -1 &  \cdots & -1 & k+1 & \cdots & k+1
    \end{array}\right],
\]
where the first block is $k\times k$. It can be shown that the principal $k\times k$ submatrix is a local maximum volume submatrix, but (see~\cref{sec:AppendixSharpness})
\[
\sigma_1(A-A_k) \geq \frac{k+2}{4}\sqrt{(m-k)(n-k)}\,\sigma_{k+1}(A).
\]
Therefore, we know that GE with local maximum volume pivoting cannot be a rank-revealer with $\mu_{m,n,k}<(k+2)\sqrt{(m-k)(n-k)}/4 = \mathcal{O}(k\sqrt{mn})$; otherwise, this matrix would violate the upper bound in~\cref{eq:GoodTrailingSV} with $j=1$. 
\label{ex:sharpness}
\end{example} 

\subsection{Computing a local maximum volume pivot in GE}\label{sec:LUalgo}
When searching for a local maximum volume pivot in GE it is important to compute the quantities of the form ${\rm vol}(\hat{B})/{\rm vol}(B)$ (see~\cref{alg:localMaxVol}), where $B$ and $\hat{B}$ are $k\times k$ submatrices of $A$ and $\hat{B}$ differs from $B$ in at most one column and at most one row. Fortunately, when ${\rm vol}(B)>0$, it turns out that this ratio can be computed efficiently. Suppose that $B = A_{11}$ is the $k\times k$ principal submatrix of $A$ after permuting the columns and rows, where
\[
	P_1AP_2 = \begin{bmatrix}
	A_{11} & A_{12} \\
	A_{21} & A_{22}
	\end{bmatrix},\qquad A_{11}\in\mathbb{R}^{k\times k}.
\]
Then, one finds that~\cite[Prop.~1]{miranian2003strong}
\begin{equation}\label{eq.volinterpreiterate}
	\frac{\vol(\hat{A}_{11})}{\vol(A_{11})} = \abs{(A_{11}^{-1}A_{12})_{st} (A_{21} A_{11}^{-1})_{ji} + (A_{11}^{-1})_{si}(S(A_{11}))_{jt}},
\end{equation}
where $S(A_{11})$ is the Schur complement in~\cref{eq.rankkLU}, and $\hat{A}_{11}$ is obtained from $A_{11} \in \R^{k \times k}$ by interchanging rows $i$ and $k+j$ and columns $s$ and $k+t$ of $P_1AP_2$. This ratio can be computed in $\mathcal{O}(k^3+mn)$ operations. For sufficiently small matrices $A_{11}^{-1}A_{12}$, $A_{21}A_{11}^{-1}$, $A_{11}^{-1}$, and $S(A_{11})$ can be computed in $\mathcal{O}(kmn)$ operations and then stored in memory. 

Despite good evidence for the path lengths from~\cref{alg:localMaxVol} being modest in practice. The worst-case behavior from~\cref{alg:localMaxVol} is still searching through all combinatorially many nodes and only finding a local maximum volume submatrix at the last possible moment. This issue motivated Miranian and Gu to consider GE with near-local maximum volume pivoting~\cite{miranian2003strong}, which can be computed more efficiently (see~\cref{sec:NearLocalMaxVol}). 

\section{QR with local maximum volume pivoting}\label{sec:sufficientQRsection}
QR with local volume pivoting is similar to CPQR, except the $k$ columns that are orthogonalized are selected as those that form a local maximum volume submatrix of $A$. Using the same notation as in~\cref{eq.rankkQR}, we find that 
\[
(A-A_k)P = \begin{bmatrix} 0 & Q_2 R_{22}\end{bmatrix},
\]
where the permutation matrix $P$ is selected so that the first $k$ columns of $AP$ are a local maximum volume submatrix (see~\cref{dfn:maxvolgeneral}). There have been decades of research studying rank-revealers based on QR and many different algorithms have been proposed (see~\cref{Tab:RRQR}). In this section, we show that QR with local maximum volume pivoting is a rank-revealer.

\begin{table}
\setlength{\tabcolsep}{5pt}
\begin{center}
\caption{Summary of deterministic QR-based approaches for estimating singular values (table adapted from~\cite{boutsidis2009improved}). In addition to algorithms in the table, there are other algorithms that consider only parts of~\cref{eq:GoodLeadingSV} and~\cref{eq:GoodTrailingSV}. These include the QR factorization algorithm in~\cite{golub1976rank}, High RRQR in~\cite{Foster1986,chan1987rank}, Low RRQR in~\cite{CH1994}, rank-revealing QR in~\cite{hong1992rank}, Hybrid I-III algorithms in~\cite{chandrasekaran1994rank}, DGEQPX and DGEQPY in~\cite{bischof1998computing}, Alg.~2-3 in~\cite{pan1999bounds}, and volume sampling-based algorithms in~\cite{deshpande2010efficient,cortinovis2020low}. Notably, as suggested by~\cite{gu1996efficient}, one can prove a $\mu_{m,n,k}$ that depends exponentially on $k$ for most of these algorithms.}
\begin{tabular}{llrll}
	\hline\\[-2ex]
	Method & $\mu_{m,n,k}$ in~\cref{eq:GoodLeadingSV}\&\cref{eq:GoodTrailingSV} & Reference &  \\ [0.5ex] 
	\hline\\[-2ex]
	CPQR & $\mathcal{O}(2^k\sqrt{n})$ &  Busin.+Golub, 1965 & \cite{GB1965} \\
	Alg.~4 & $\mathcal{O}(\sqrt{kn})$ & Gu+ Eisenstat, 1996 & \cite{gu1996efficient} \\ 	
	Alg.~1 & $\mathcal{O}(\sqrt{kn})$ &  Pan+Tang, 1999 & \cite{pan1999bounds} \\ 
	Alg.~2 & $\mathcal{O}(\sqrt{kn})$ &  Pan, 2000 & \cite{pan2000existence} \\ 
	\hline
\end{tabular}
\label{Tab:RRQR}
\end{center}
\end{table}

Despite appearances, GE and QR have a well-known connection via the Cholesky decomposition. In particular, from~\cref{eq.rankkQR}, we have the partial Cholesky factorization
\begin{equation} 
	P^\top A^\top A P = 
	\begin{bmatrix}
		R_{11}^\top & 0 \\
		R_{12}^\top & R_{22}^\top
	\end{bmatrix}
	\begin{bmatrix}
		R_{11} & R_{12} \\
		0 & R_{22}
	\end{bmatrix},
\label{eq:PartialCholesky}
\end{equation} 
which comes from using the fact that $\begin{bmatrix}Q_1 & Q_2 \end{bmatrix}$ has orthonormal columns. Therefore, QR with local maximum volume pivoting on $A$ is closely related to GE with local maximum volume pivoting on $A^\top A$. The only minor difference is that GE can permute the columns and rows separately, while QR can only permute the columns of $A$. This does not matter in the end because if $A(:,\mathcal{J})$ is a $m\times k$ local maximum volume submatrix of $A$, then $A(:,\mathcal{J})^\top A(:,\mathcal{J})$ is a $k\times k$ local maximum volume submatrix of $A^\top A$. To see this, let $\hat{\mathcal{J}}$ and $\tilde{\mathcal{J}}$ be two sets of indices of size $k$ that differ from $\mathcal{J}$ in at most one entry. Then, we have 
\[
\begin{aligned} 
{\rm vol}\left(A(:,\hat{\mathcal{J}})^\top A(:,\tilde{\mathcal{J}})\right) & = \prod_{j=1}^k \sigma_j\left(A(:,\hat{\mathcal{J}})^\top A(:,\tilde{\mathcal{J}})\right)\\
& \leq \prod_{j=1}^k \sigma_j\left(A(:,\hat{\mathcal{J}})\right) \sigma_j\left( A(:,\tilde{\mathcal{J}})\right)\\
& \leq \prod_{j=1}^k \sigma_j^2\left(A(:,\mathcal{J})\right) = {\rm vol}\left(A(:,\mathcal{J})^\top A(:,\mathcal{J})\right), \\
\end{aligned} 
\]
where the last inequality follows from the assumption that $A(:,\mathcal{J})$ is a $m\times k$ local maximum volume submatrix of $A$.  A local maximum volume submatrix in QR also ensures that the resulting partial QR factorization satisfies an interpolative bound with $\nu = 1$. 

\subsection{QR with local maximum volume pivoting is a rank-revealer}\label{sec:QRrankrevealer}
We now show that QR with local maximum volume pivoting is a rank-revealer. The result is related to a statement proved by Gu and Eisenstat on QR with local maximum volume pivoting~\cite[Thm.~3.2]{gu1996efficient}. We prove it from~\cref{thm.sufficientLU}. 

\begin{theorem}\label{thm.sufficientQR}
QR with local maximum volume pivoting is a rank-revealer with $\mu_{m,n,k} = \sqrt{1+5kn}$ (see~\cref{eq:GoodLeadingSV} and~\cref{eq:GoodTrailingSV}) and computes a partial QR factorization satisfying an interpolative bound with $\nu\leq 1$ (see~\cref{def:InterpolativeBoundsQR}).
\end{theorem}
\begin{proof}
We may assume that $k\leq {\rm rank}(A)$ otherwise a local maximum volume pivot does not exist. Since a local maximum volume $m\times k$ submatrix of $A$ is also a local maximum volume $k\times k$ submatrix of $A^\top A$, there is a close connection between QR on $A$ and GE on $A^\top A$ with local maximum volume pivoting. In particular, if $A_k$ is a rank $k$ approximant of $A$ formed by QR with local maximum pivoting on $A$, then $A_k^\top A_k$ is a rank $k$ approximation to $A^\top A$ that GE could construct with local maximum volume pivoting on $A^\top A$. From~\cref{thm.sufficientLU} with $m=n$, we have
\[
\frac{1}{1+5kn}\sigma_j(A^\top A)\leq \sigma_j(A_k^\top A_k) \leq \left(1+5kn\right)\sigma_j(A^\top A), \quad 1\leq j\leq k.
\]
as well as 
\[
\frac{1}{1+5kn}\sigma_{k+j}(A^\top A)\leq \sigma_j(A^\top A-A_k^\top A_k) \leq \left(1+5kn\right)\sigma_{j+k}(A^\top A), \quad 1\leq j\leq n-k.
\]
The statement of the theorem follows by noting that $\sigma_j(A^\top A) = \sigma_j^2(A)$, $\sigma_j(A_k^\top A_k) = \sigma_j^2(A_k)$, and $\sigma_j(A^\top A-A_k^\top A_k) = \sigma_j^2(A-A_k)$ since 
\[
P^\top\!(A^\top \!\!A\!-\!A_k^\top\! A_k)P \!= \!\!\begin{bmatrix}0 \!\!& \!\!0 \\ 0 \!\!& \!\!R_{22}^\top R_{22} \end{bmatrix} \!\!=\!\! \begin{bmatrix}0 \\ R_{22}^\top Q_2^\top \end{bmatrix}\!\!\begin{bmatrix}0 \!&\! Q_2R_{22} \end{bmatrix} \!\!= \!((A-A_k)P)^\top\!((A-A_k)P). 
\] 
\end{proof}

\Cref{thm.sufficientQR} shows that QR with local maximum volume pivoting is a singular value estimator of $A$ to within a factor of $\sqrt{1+5kn}$. It is worth noting that this factor is independent of $m$, so QR is an ideal rank-revealer for a tall-skinny matrix.  For any $k$, $A_k$ estimates the first $k$ singular values of $A$ and $\|A-A_k\|_2\leq \sqrt{1+5kn}\sigma_{k+1}(A)$. The $\mu_{m,n,k}$ factor for QR with local maximum volume pivoting is significant better than that for GE. With $m=n=1000$ and $k=10$, the GE factor is about $3\times 10^4$ while the QR factor is under $200$. It is reasonable to expect that QR with local maximum volume pivoting is a slightly better singular value estimator in practice, though typically the decision to select GE or QR depends on whether it is convenient to have a one-sided or two-sided interpolative decomposition. 

The following example illustrates that~\cref{thm.sufficientQR} is essentially sharp in the sense that the $\mu_{m,n,k}$ factor in~\cref{thm.sufficientQR} can only be improved by reducing its constant terms.
\begin{example}
Consider the following odd-looking $(k+1)\times n$ matrix: 
\[
A =\sqrt{k+2}\left[\begin{array}{ccccc|ccc}
        d_1 & a_1 & a_1 & \cdots & a_1& a_1 & \cdots & a_1 \\
         & d_2 & a_2 & \cdots & a_2& a_2 & \cdots & a_2 \\
         & & \ddots & \ddots & \vdots & \vdots & \ddots & \vdots \\
         & & & d_{k-1} & a_{k-1} & a_{k-1} & \cdots & a_{k-1} \\
         & & & & d_k& a_{k} & \cdots & a_{k} \\
         & & & & & d_{k+1} & \cdots & d_{k+1} \\
    \end{array}\right],
\]
where $d_i = \sqrt{k-i+2}/\sqrt{k-i+3}$ and $a_i = -1/\sqrt{(k-i+2)(k-i+3)}$. The reason for such an odd-looking matrix is that $A^\top A$ is the same as the matrix in~\cref{ex:sharpness} with $m=n$. Due to the relationship between GE and QR, we immediately conclude that QR with local maximum volume cannot be a rank-revealer with $\mu_{m,n,k}<\sqrt{(k+2)(n-k)/4} = \mathcal{O}(\sqrt{kn})$. 
\end{example} 

\subsection{Computing a local maximum volume subset in QR factorization}\label{sec:QRalgo}
When searching for a local maximum volume pivot in QR, the submatrix is $m\times k$, so no row swaps are considered. Thus, in~\cref{alg:localMaxVol}, we need to compute quantities of the form ${\rm vol}(\hat{B})/{\rm vol}(B)$, where $B$ and $\hat{B}$ only differ in at most one column. This ratio can be computed very efficiently. Suppose that $B$ are the first $k$ columns of $AP$, where
\[
AP = \begin{bmatrix}B & C \end{bmatrix} = \begin{bmatrix}Q_1 & Q_2 \end{bmatrix} \begin{bmatrix}R_{11} & R_{21} \\ 0 & R_{22} \end{bmatrix}. 
\]
Then, one finds that~\cite[Lem.~3.1]{gu1996efficient}
\begin{equation} 
\frac{\vol(\hat{B})}{\vol(B)} = \sqrt{\big(R_{11}^{-1} R_{12}\big)^2_{ij} \!+\! \big((R_{11}^\top R_{11})^{-1}\big)_{ii} \big(R_{22}^\top R_{22}\big)_{jj}},
\label{eq:UpdateQR}
\end{equation} 
where $\hat{B}$ is obtained from $B$ by interchanging columns $i$ and $k+j$ of $AP$.  

The pitfalls with maximum volume pivoting for QR closely follow those for GE. There is still the potential for extremely long paths in~\cref{alg:localMaxVol} before it locates a local maximum volume submatrix. To overcome this issue, we describe the idea of near-local maximum volume pivoting in the following section.  

\section{Near-local maximum volume pivoting}\label{sec:NearLocalMaxVol}

To overcome the deficiencies with local maximum volume pivoting, the concept of local maximum volume can be weakened to near-local maximum volume. Given a relaxation parameter $\gamma>1$ (e.g., $\gamma=2$), one searches for a $p\times q$ submatrix $B$ such that 
\begin{equation} 
{\rm vol}(B) \geq \frac{1}{\gamma} {\rm vol}(\hat{B})
\label{eq:NearLocalMaxVol}
\end{equation} 
for all $p \times q$ submatrices $\hat{B}$ of $A$ that differ from $B$ in at most one row and at most one column. If a submatrix $B$ satisfies~\cref{eq:NearLocalMaxVol} then we say it is a $\gamma$-local maximum volume submatrix. The larger the value of $\gamma$, the easier it is to find such a submatrix but the weaker the final singular values estimates from GE and QR might be. 

\Cref{alg:localMaxVol} only needs to be slightly modified to allow one to compute a $\gamma$-local maximum volume pivot with parameter $\gamma>1$ (see~\cref{alg:NearLocalMaxVol}). There are two main changes: (i) one should start at any submatrix of $A$ that has a reasonably large volume.\footnote{One can also start~\cref{alg:localMaxVol} at an initial submatrix with a reasonably large volume and this may improve its pathlength in practice. However, since~\cref{alg:localMaxVol} is seeking a local maximum volume submatrix, one still cannot bound its pathlengths by a polynomial in $k$, $m$, and $n$.} For GE, we recommend selecting the submatrix found by GECP and for QR one can select the submatrix from CPQR, and (ii) one moves to a neighboring node in the volume submatrix graph (see~\cref{sec:volumegraph}) only if there is more than a factor of $\gamma$ increase in volume. With such details, a near-local maximum volume submatrix can be found in a computationally efficient way (see~\cref{sec:LocalMaxVolQR,sec:LocalMaxVolGE}).  When $\gamma>2$ and for any $A_{11}$, there is a quick look-up for all $(k(m-k)+1)(k(n-k)+1)-1 = \mathcal{O}(k^2mn)$ possible neighboring $\hat{A}_{11}$'s~\cite{schork2020rank}. So, the last step of~\cref{alg:NearLocalMaxVol} that requires checking all possible neighbors is not a computational bottleneck. 

\begin{algorithm}
\caption{Finding a $\gamma$-local maximum volume $p\times q$ submatrix of $A$ with $\gamma>1$.}
\label{alg:NearLocalMaxVol}
\begin{algorithmic}[1]
\State Start at a greedily-selected $p\times q$ submatrix $B$ of $A$ and compute ${\rm vol}(B)$.  
\For{each neighbor $\hat{B}$ of $B$ (see~\cref{sec:volumegraph})}
\State Compute ${\rm vol}(\hat{B})/{\rm vol}(B)$ (see~\cref{eq.volinterpreiterate,eq:UpdateQR} for efficient formulas).
\If{${\rm vol}(\hat{B})/{\rm vol}(B)>\gamma$}
\State Replace $B$ with $\hat{B}$ and go back to step 2. 
\EndIf
\EndFor
\State The algorithm terminates with a near-local maximum volume submatrix $B$. 
\end{algorithmic}
\end{algorithm}

\subsection{Near-local maximum volume pivoting in GE}\label{sec:LocalMaxVolGE}
For GE with near-local maximum volume pivoting, we recommend that~\cref{alg:NearLocalMaxVol} starts with an initial $k\times k$ submatrix selected by GECP. We now show that this gives a submatrix that keeps the pathlengths in~\cref{alg:NearLocalMaxVol} short. 

\subsubsection{Volume of GECP pivot} To bound the longest possible pathlength with the initial submatrix coming from GECP, we need to lower bound the volume of the submatrix selected by GECP.

\begin{lemma}
Let $A_{11} \in \mathbb{R}^{k \times k}$ be a submatrix of $A \in \mathbb{R}^{m \times n}$ selected by GECP. Then, 
\[
{\rm vol}(A_{11}) \geq \frac{1}{4^{k+1} (k+\rho_{k+1}) \sqrt{(m-k)(n-k)}} \prod_{j=1}^k \sigma_j(A),
\]
where $\rho_{k+1}$ is the growth rate GECP on a $(k+1)\times (k+1)$ matrix~\cite{wilkinson1961error}.
\label{lem:GECPvolume}
\end{lemma} 
\begin{proof}
Consider the partial LU factorization in~\cref{eq.rankkLU} and factorize $A_{11}$ into an LDU decomposition, i.e., $A_{11} = LDU$. Since $A_{11}$ is selected by GECP, we must have
    \[
        \left\|D^{-1}L^{-1}\begin{bmatrix}A_{11} & A_{12}\end{bmatrix}\right\|_{\max} = \left\|\begin{bmatrix} U & D^{-1}L^{-1}A_{12}\end{bmatrix}\right\|_{\max} \leq 1
    \]
and
    \[
        \left\|\begin{bmatrix}A_{11} \\ A_{21}\end{bmatrix}U^{-1}D^{-1}\right\|_{\max} = \left\|\begin{bmatrix}L \\ A_{21}U^{-1}D^{-1}\end{bmatrix}\right\|_{\max} \leq 1.
    \]
    Since $L$ and $U$ are $k\times k$ triangular matrices with ones on the diagonal entries and $\|L\|_{\max}\leq 1$ and $\|U\|_{\max}\leq 1$, we find that $\|L^{-1}\|_F\leq 2^k$ and $\|U^{-1}\|_F \leq 2^k$~\cite[p.~107]{faddeev1970solution}. Hence, we have
    \begin{equation}\label{eq.interpbdrightcomplete}
        \|A_{11}^{-1}A_{12}\|_F \leq \|U^{-1}\|_F \|D^{-1}L^{-1}A_{12}\|_F \leq 2^k \sqrt{k(n-k)}
    \end{equation}
    and by a similar argument, we have $\|A_{21}A_{11}^{-1}\|_F\leq 2^k \sqrt{k(m-k)}$. Next, we have
    \begin{equation} 
        \frac{\sigma_1(S(A_{11}))}{\sigma_k(A_{11})} = \|S(A_{11})\|_2 \|A_{11}^{-1}\|_2 \leq \frac{\sqrt{(m-k)(n-k)}\|S(A_{11})\|_{\max}}{\min |\text{diag}(D)|} \|L^{-1}\|_F \|U^{-1}\|_F,
        \label{eq:ratiobounding}
    \end{equation}
    where we used the norm equivalence between the max norm and $2$-norm. Let $j$ be an entry that minimizes $|D(j,j)|$. Then, we have $\min |\text{diag}(D)| = |D(j,j)| = \|S(A(1\!:\!(j\!-\!1),1\!:\!(j\!-\!1)))\|_{\max}$, where we take $S(A(1\!:\!(j\!-\!1),1\!:\!(j\!-\!1))) = A$ if $j = 1$. But $S(A_{11})$ is also the $(k-j+1)$th Schur complement of $S(A(1\!:\!(j\!-\!1),1\!:\!(j\!-\!1)))$ due to complete pivoting. Hence, we have
    \[
        \frac{\|S(A_{11})\|_{\max}}{\min |\text{diag}(D)|} \leq \rho_{k+1},
    \]
    where $\rho_{k+1}$ is the growth rate of GECP on a $(k+1)\times (k+1)$ matrix. Thus, from~\cref{eq:ratiobounding}, we find that $\sigma_1(S(A_{11}))/\sigma_k(A_{11}) \leq \sqrt{(m-k)(n-k)}\rho_{k+1}4^k$. Now, if $\sigma_j(A_{11}) > 0$, by~\cite{miranian2003strong},~\cref{eq.interpbdrightcomplete},~$\|A_{21}A_{11}^{-1}\|_F\leq 2^k \sqrt{k(m-k)}$, and $\sigma_1(S(A_{11}))/\sigma_k(A_{11}) \leq \sqrt{(m-k)(n-k)}\rho_{k+1}4^k$, we have
    \begin{align*}
        \frac{\sigma_j(A)}{\sigma_j(A_{11}\!)}\! &\leq \sqrt{1 \!+\! \|A_{21}A_{11}^{-1}\|_F^2 \!+\! \sqrt{\frac{m-k}{n-k}}\frac{\sigma_1(S(A_{11}))}{\sigma_k(A_{11})}} \times \\ 
          &\qquad  \qquad  \qquad\qquad\qquad  \qquad\qquad  \sqrt{1 \!+\! \|A_{11}^{-1}A_{12}\|_F^2 \!+\! \sqrt{\frac{n-k}{m-k}}\frac{\sigma_1(S(A_{11}))}{\sigma_k(A_{11})}} \\
        &\leq \sqrt{1 + 4^kk(m-k) + 4^k\rho_{k+1}(m-k)}\sqrt{1 + 4^kk(n-k) + 4^k\rho_{k+1}(n-k)} \\
        &\leq 4^{k+1} (k+\rho_{k+1}) \sqrt{(m-k)(n-k)}.
    \end{align*}
The result follows as ${\rm vol}(A_{11}) = \prod_{j=1}^k \sigma_j(A_{11})$. 
\end{proof}

The growth rate for GECP can in principle be extremely large. It is known that $\rho_k\leq k^{1/2}\left(2\cdot3^{1/2}\cdots n^{1/(n-1)}\right)^{1/2}\sim c n^{1/2} n^{\tfrac{1}{4}\log n}$ for some small constant $c$~\cite{wilkinson1961error}. However, this bound is very pessimistic and it is non-trivial to even find matrices with $\rho_k>k$~\cite{edelman2023some}. \Cref{lem:GECPvolume} tells us that the volume of a submatrix selected by GECP could be {\em much} smaller than the product of the first $k$ singular values of $A$. Such a submatrix may seem too poor to use as an initial guess; however, it is important to realize that (1)~\cref{alg:NearLocalMaxVol} increases the volume of a submatrix exponentially in terms of pathlengths (i.e., after $\ell$ steps the volume has increased by at least a factor of $\gamma^\ell$) and (2) \cref{lem:GECPvolume} is extremely pessimistic and we observe that GECP selects much better initial submatrices in practice. 

Doing GECP for $k$ steps is also a very efficient algorithm. The first step of GECP on an $m \times n$ matrix requires: (1) $m n - 1$ comparisons to find the maximum entry in absolute value and (2) for each step we use $(m-1)(2n-1)$ flops to do the elimination. For $k$ steps, remembering that one column and row are being eliminated after each step, gives a total of
\[
\begin{aligned}
	\sum_{j=1}^k (m\!-\!j+1)(n\!-\!j+1) \!- \!1 &+ \!\sum_{j=1}^k (m\!-\!j)(2(n\!-\!j)\!+\!1)  \leq 3\!\sum_{j=1}^k \!(m\!-\!j+1)(n\!-\!j+1)\\
	& =\frac{1}{6}k\left(2k^2 - 3k(m+n+1) + m(6n+3) + 3n +1\right)
\end{aligned}
\]
floating-point operations. 

\subsubsection{Using GECP for finding a near-local maximum volume pivot}
The volume of the $k\times k$ submatrix is increased by a factor of at least $\gamma^\ell$ after $\ell$ steps in~\cref{alg:NearLocalMaxVol}. The largest possible volume of any $k\times k$ submatrix of $A$ is $\prod_{j=1}^k \sigma_j(A)$ so we the value of $\ell$ must satisfy
\[
\frac{1}{4^{k+1} (k+\rho_{k+1}) \sqrt{(m-k)(n-k)}}\prod_{j=1}^k \sigma_j(A) \gamma^\ell \leq \prod_{j=1}^k \sigma_j(A). 
\]
In other words, we know that $\ell \leq (k+1)\log_\gamma(4) + \log_\gamma(k+\rho_{k+1})+\log_\gamma(n-k)/2+ \log_\gamma(m-k)/2$. Therefore, the path lengths from~\cref{alg:NearLocalMaxVol} are always modest.

In~\cref{fig:LUexperiments}, we show that GE with $3$-local maximum volume pivoting is only within a factor of $1.4$ slower than GECP on a $500\times 500$ matrix. 
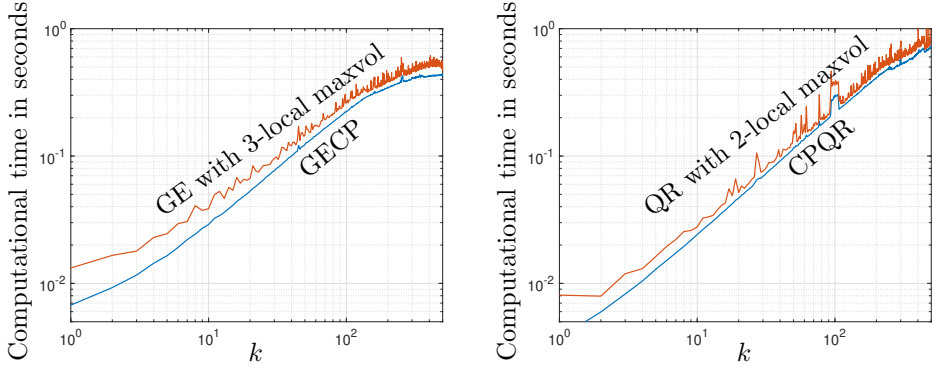
\begin{figure} 
\centering
\begin{minipage}{.49\textwidth}
\begin{overpic}[width=\textwidth]{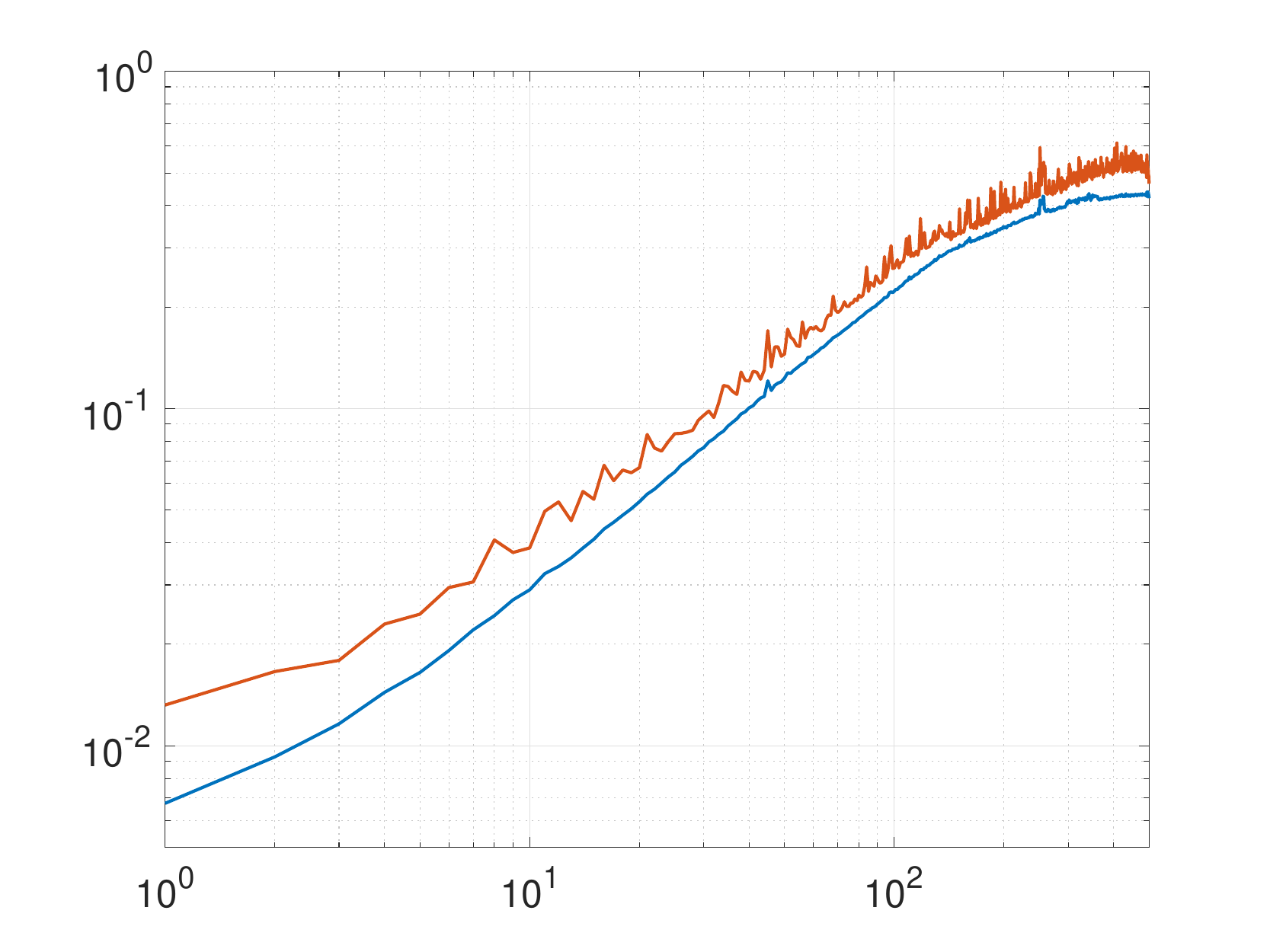}
\put(50,0) {$k$} 
\put(0,0) {\rotatebox{90}{Computational time in seconds}}
\put(60,38) {\rotatebox{40}{GECP}}
\put(30,30) {\rotatebox{36}{GE with $3$-local maxvol}}
\end{overpic} 
\end{minipage}
\begin{minipage}{.49\textwidth}
\begin{overpic}[width=\textwidth]{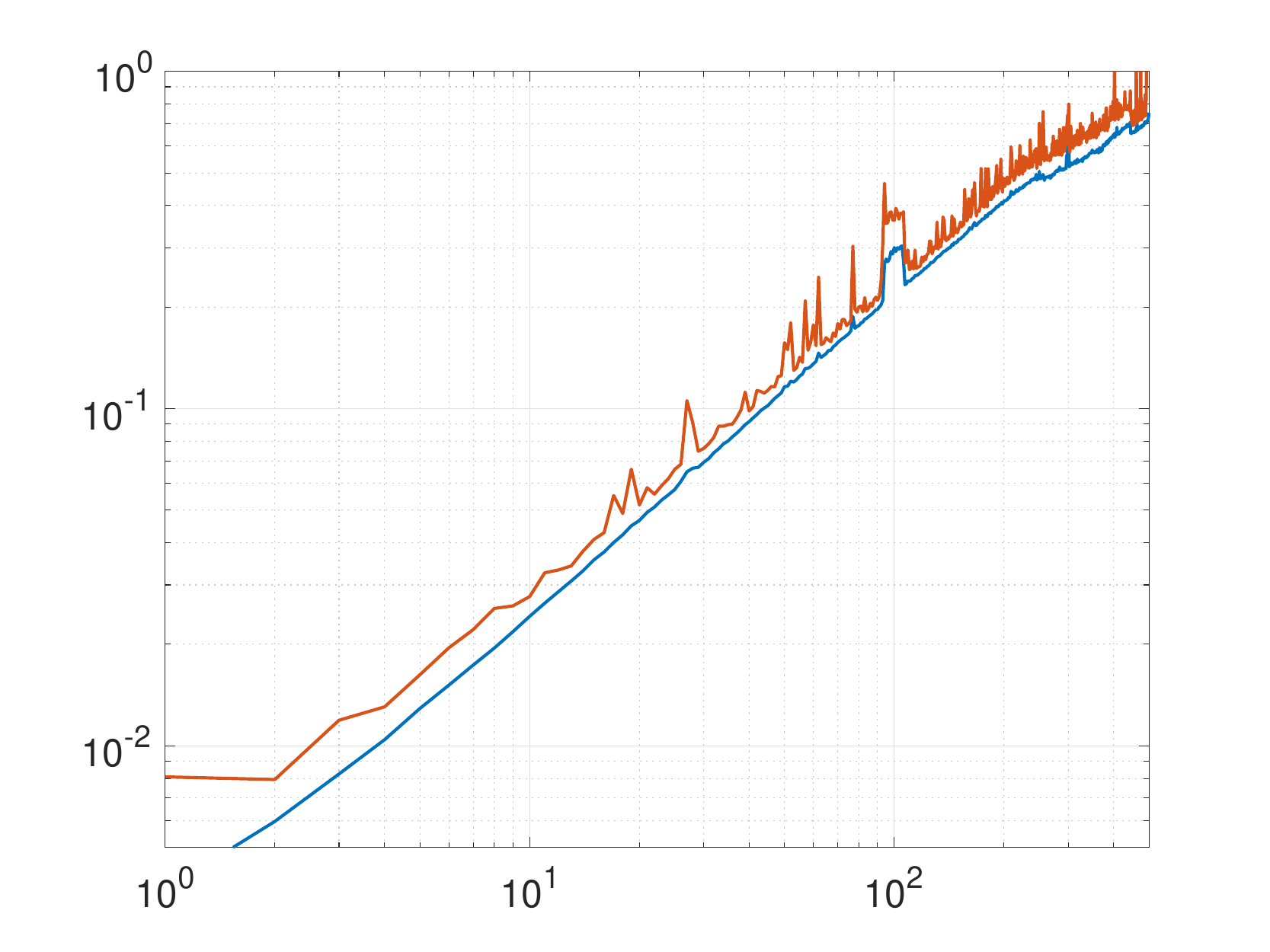}
\put(50,0) {$k$} 
\put(0,0) {\rotatebox{90}{Computational time in seconds}}
\put(60,38) {\rotatebox{40}{CPQR}}
\put(30,30) {\rotatebox{38}{QR with $2$-local maxvol}}
\end{overpic} 
\end{minipage}
\caption{The computational cost of GE and QR for $1\leq k\leq 500$ with near-local maximum volume pivoting on a $500\times 500$ randomly generated matrix with standard Gaussian entries. Here, we are finding a near-local maximum with~\cref{alg:NearLocalMaxVol}. Left: For GE, we select $\gamma=3$ and compare the timings against GECP. We find that the computational cost of GE with $3$-local maximum volume pivoting is no more than $1.4\times$ slower than GECP. A similar observation is made in~\cite{schork2020rank}. Right: For QR, we select $\gamma=2$ and compare the timings against our implementation of CPQR. We find that the computational cost of QR with $2$-local maximum volume pivoting is no more than $2\times$ slower than CPQR.}
\label{fig:LUexperiments} 
\end{figure}

\subsubsection{GE with near-local maximum volume pivoting is a rank-revealer}
The proof of~\cref{thm.sufficientQR} can be easily adapted to allow for near-local maximum volume pivoting. 
\begin{theorem}\label{thm.sufficientLUnear}
Let $\gamma>1$.  GE with $\gamma$-local maximum volume pivoting is a rank-revealer with $\mu_{m,n,k} = 1+5\gamma^2 k\sqrt{mn}$ (see~\cref{eq:GoodLeadingSV} and~\cref{eq:GoodTrailingSV}) and computes a partial LU factorization satisfying interpolative bounds with $\nu\leq\gamma$ (see~\cref{def:InterpolativeBoundsGE}). 
\end{theorem}
\begin{proof}
The proof closely follows the argument for in~\cref{thm.sufficientLU}. The only difference is that when~\cref{eq:NearLocalMaxVol} holds with $\gamma \geq 1$,~\cite[Thm.~2]{gu1996efficient} shows that~\cref{eq:FirstInequality} and~\cref{eq:eq2} become
	\begin{equation} 
		\begin{aligned}
		\sigma_j(A_{11}) \!\leq\! \sigma_j(A) &\leq \left(\!2\gamma^2 k\sqrt{(m\!-\!k)(n\!-\!k)} \!+\! \sqrt{(1\!+\! \gamma^2 k(n\!-\!k)\!)(1\!+\! \gamma^2 k(m\!-\!k))}\right)\!\sigma_j(A_{11}) \\
		&\leq \left(1+5\gamma^2k\sqrt{mn}\right)\sigma_j(A_{11}), \quad 1\leq j\leq k,
		\end{aligned}
		\label{eq:FirstInequalitynear}
		\end{equation} 
		and
		 \begin{equation} 
 		\begin{aligned}
		&\sigma_{k+j}(A) \leq \sigma_j(S(A_{11})) \\
		&\qquad \leq \left(2\gamma^2 k\sqrt{(m\!-\!k)(n\!-\!k)} \!+\!\! \sqrt{(1\!+\!\gamma^2 k(n\!-\!k))(1\!+\!\gamma^2 k(m\!-\!k))}\right)\! \sigma_{k+j}(A) \\
		&\qquad\leq \left(1+5\gamma^2k\sqrt{mn}\right)\sigma_{k+j}(A), \quad 1\leq j\leq \min(m,n)\!-\!k,
		\end{aligned}
		\label{eq:eq2near} 
	\end{equation}
	respectively.
\end{proof}

\subsection{Near-local maximum volume pivoting in QR}\label{sec:LocalMaxVolQR}
For QR, we recommend that~\cref{alg:NearLocalMaxVol} starts with an initial $m\times k$ submatrix selected by CPQR. Our suggestion here is very similar to~\cite[Alg.~4]{gu1996efficient}, except they select $\gamma$ to be a small power of $n$ such as $\gamma = \sqrt{n}$, while we insist on the fact that $\gamma$ should be a constant (e.g., $\gamma=2$) independent of $n$ (see~\cref{sec:AfterTheorem}).  

Suppose that the initial submatrix selected by CPQR is denoted by $B$, then we know that~\cite[Thm.~7.2]{gu1996efficient}
\[
{\rm vol}(B) = \prod_{j=1}^k \sigma_j(B) \geq \frac{1}{2^k\sqrt{n-k}}\prod_{j=1}^k \sigma_j(A). 
\]
This volume is increased by a factor of at least $\gamma^\ell$, if the path length of~\cref{alg:NearLocalMaxVol} is $\ell$. Since the matrix $B$ in~\cref{alg:NearLocalMaxVol} is always a submatrix of $A$, its largest possible volume is $\prod_{j=1}^k \sigma_j(A)$. Hence, the value of $\ell$ must satisfy
\[
\frac{1}{2^k\sqrt{n-k}}\prod_{j=1}^k \sigma_j(A) \gamma^\ell \leq \prod_{j=1}^k \sigma_j(A). 
\]
In other words, we know that $\ell\leq \log_{\gamma}(2^k\sqrt{n-k}) = k\log_\gamma(2)+\log_\gamma(n-k)/2$. Therefore, the path lengths from~\cref{alg:NearLocalMaxVol} are always modest. In~\cref{fig:LUexperiments}, we show that QR with $2$-local maximum volume pivoting is only within a factor of $2$ slower than CPQR on a $500\times 500$ matrix.

At each node in the volume submatrix graph, one may have to consider all its $k(n-k)$ neighbors before finding a node with a $\gamma$ factor increase. From here, one can carefully calculate the cost of computing the partial QR factorization to find that it requires $\mathcal{O}(kmn\log_{\gamma}(2) + mn\log_{\gamma}(n))$ operations~\cite[Sec.~4.4]{gu1996efficient}.

\subsubsection{QR with near-local maximum volume pivoting is a rank-revealer}\label{sec:AfterTheorem} 

The proof of~\cref{thm.sufficientQR} can be easily adapted to allow for near-local maximum volume pivoting. We state it for completeness.  
\begin{theorem}\label{thm.sufficientQRnear}
Let $\gamma> 1$. QR with $\gamma$-local maximum volume pivoting is a rank-revealer with $\mu_{m,n,k} = \sqrt{1+5\gamma^2 kn}$ (see~\cref{eq:GoodLeadingSV} and~\cref{eq:GoodTrailingSV}) and computes a partial QR factorization satisfying an interpolative bound with $\nu\leq \gamma$ (see~\cref{def:InterpolativeBoundsQR}). 
\end{theorem}
\begin{proof}
The proof of this closely follows the proof of~\Cref{thm.sufficientQR}, except one should use~\Cref{thm.sufficientLUnear}.
\end{proof}

In~\cite{gu1996efficient} they select $\gamma$ to be a small power of $n$ such as $\gamma = \sqrt{n}$
which results in an algorithmic complexity of $\mathcal{O}(kmn)$ but $\mu_{m,n,k} = \mathcal{O}(\sqrt{k}n)$. Since CPQR delivers $\mu_{m,n,k} = \mathcal{O}(2^k\sqrt{n})$---which is asymptotically better in $n$---and near-local maximum volume is efficient to compute, we advocate for $\gamma$ to be a small constant (e.g., $\gamma=2$). QR with near-local maximum volume pivoting and $\gamma=2$ has $\mu_{m,n,k} = \mathcal{O}(\sqrt{kn})$, which matches CPQR's factor in $n$ and is much better in $k$. 

\section{Near-local maximum volume pivoting is necessary}\label{sec:Necessary}
One may wonder if there are pivoting strategies, other than near-local maximum volume pivoting, that ensure that GE and QR are rank-revealers. It turns out that any partial LU or QR factorization that can be used to estimate singular values and satisfy interpolative bounds must have found a near-local maximum volume submatrix of $A$ (see~\cref{eq:NearLocalMaxVol}). Therefore, the only pivoting strategy in GE and QR that ensure they are rank-revealers is  near-local maximum volume pivoting. 

\subsection{All GE-based rank-revealers use near-local maximum volume pivoting}
\label{sec:NecessaryLU}
We start by considering the partial LU factorization of $A$ given in~\cref{eq.rankkLU}. We show that provided three conditions hold: (1) $\sigma_k(A_{11})\approx \sigma_k(A)$, (2) $\|S(A_{11})\|_2\approx \sigma_{k+1}(A)$, and (3) the factorization satisfies interpolative bounds (see~\cref{def:InterpolativeBoundsGE}), then $A_{11}$ is a near-local maximum volume $k\times k$ submatrix of $A$. Notably, the first two conditions are actually slightly weaker than~\cref{eq:GoodTrailingSV,eq:GoodLeadingSV} since they only consider the $\sigma_k$ and $\sigma_{k+1}.$

\begin{theorem}\label{thm.necessaryLU}
Let $A\in\mathbb{R}^{m\times n}$ and $1\leq k\leq {\rm rank}(A)$. Any partial LU factorization of $A$ in~\cref{eq.rankkLU} that satisfies, for some $\mu,\nu\geq 1$,
\begin{enumerate} 
\item $\sigma_k(A) \leq \mu \sigma_k(A_{11})$, 
\item $\|S(A_{11})\|_2 \leq \mu \sigma_{k+1}(A)$, and 
\item $\|A_{21}A_{11}^{-1}\|_{\max}\leq \nu$ and $\|A_{11}^{-1}A_{12}\|_{\max} \leq \nu$,
\end{enumerate} 
must have $A_{11}$ as a near-local maximum volume $k\times k$ submatrix of $A$ with $\gamma \leq \nu^2+\mu^2$.
\end{theorem}
\begin{proof}
In~\cref{eq.volinterpreiterate}, there is a formula for ${\rm vol}(\hat{A}_{11})/{\rm vol}(A_{11})$, where $\hat{A}_{11}$ is obtained from $A_{11}$ by interchanging rows $i$ and $k+j$ and columns $s$ and $k+t$. Applying the triangle inequality to that formula, we find that for any interchanges we have
\[
\left|\frac{\vol(\hat{A}_{11})}{\vol(A_{11})}\right| \leq \nu^2 + \|A_{11}^{-1}\|_{\max}\|S(A_{11})\|_{\max}
\leq \nu^2 + \frac{\|S(A_{11})\|_2}{\sigma_k(A_{11})},
\]
where we used that $\|B\|_{\max}\leq \|B\|_2$ for any matrix $B$ and $\|A_{11}^{-1}\|_2 = 1/\sigma_{k}(A_{11})$. The result follows as 
\[
\frac{\|S(A_{11})\|_2}{\sigma_k(A_{11})} \leq \frac{\mu\sigma_{k+1}(A)}{\sigma_k(A)/\mu} = \mu^2\frac{\sigma_{k+1}(A)}{\sigma_k(A)}\leq \mu^2. 
\]
\end{proof}

One may wonder if the factor of $ \nu^2 + \mu^2$ in~\cref{thm.necessaryLU} can be improved. The following two examples show that it cannot be improved to anything $<\max(\nu^2,\mu^2)$. 
\begin{example} 
Fix $k = 2$ and consider the following $4\times 4$ matrix:
\[
	A = \left[
	\begin{array}{cc|cc}
		1  & 0 & 0  & 0\\
    		0  & \mu^{-1}  & 0  & 0\\
		  \hline
    		0  & 0  & \mu  & 0\\
    		0  & 0  & 0  & 1
	\end{array}
	\right], \qquad \mu>1.
\]
Here, $A_{11}$ is the principal $2\times 2$ block. It is simple to check that the three conditions of~\cref{thm.necessaryLU} are satisfied with $\nu=0$: (1) $\sigma_2(A) = 1 \leq \mu \mu^{-1} = \mu\sigma_2(A_{11})$, (2) $\|S(A_{11})\|_2= \mu\leq \mu = \mu\sigma_3(A)$, and (3) $\|A_{21}A_{11}^{-1}\|_{\max}= \|A_{11}^{-1}A_{12}\|_{\max}=0$. However, the volume of $A([1,2],[1,2])$ is $\mu^{-1}$ while the volume of $A([1,3],[1,3])$ is $\mu$.   Therefore, $A_{11}$ is only a near-local maximum volume pivot with $\gamma = \mu^2$. We conclude that~\cref{thm.necessaryLU} does not hold if the factor $\nu^2+\mu^2$ is reduced to $<\mu^2$.
\label{ex:FirstOne}
\end{example} 

A different example shows that the factor of $\nu^2+\mu^2$ in~\cref{thm.necessaryLU} cannot be reduced to $<\nu^2$. 

\begin{example} 
Fix $k = 3$ and consider the following $5\times 5$ matrix:
\[
	A = \left[
	\begin{array}{ccc|cc}
		1  & 0 & 0  & 1 & 0 \\
		0 & 1 & 0 & \nu & 0 \\ 
		0 & 0 & 1 & 0 & 0 \\
		\hline
		-\nu & 1 & 0 & 0 & 0 \\
		0 & 0 & 0 & 0 & 1
	\end{array}
	\right], \qquad \nu>1.
\]
Here, $A_{11}$ is the principal $3\times 3$ identity block. Since the $4 \times 4$ identity matrix is a submatrix of $A$, we must have $\sigma_3(A)\geq 1$ and $\sigma_4(A) \geq 1$. On the other hand, 
\[
	A_2 = \left[
	\begin{array}{ccc|cc}
		1  & 0 & 0  & 1 & 0 \\
		0 & 1 & 0 & \nu & 0 \\ 
		0 & 0 & 0 & 0 & 0 \\
		\hline
		-\nu & 1 & 0 & 0 & 0 \\
		0 & 0 & 0 & 0 & 0
	\end{array}
	\right]
\]
is a rank-$2$ matrix with $\|A - A_2\|_2 = 1$. Hence, we have $\sigma_3(A) = \sigma_4(A) = 1$. The three conditions of~\cref{thm.necessaryLU} are satisfied with $\mu=1$: (1) $\sigma_3(A) = 1 \leq 1 = \sigma_3(A_{11})$, (2) $S(A_{11})$ is the $2\times 2$ identity so $\|S(A_{11})\|_2 = 1 \leq 1 = \sigma_4(A)$, and (3) $\|A_{21}A_{11}^{-1}\|_{\max}= \|A_{11}^{-1}A_{12}\|_{\max}=\nu$. However, the volume of $A([1,2,3],[1,2,3])$ is $1$ while the volume of $A([2,3,4],[1,3,4])$ is $\nu^2$.   Therefore, $A_{11}$ is only a near-local maximum volume pivot with $\gamma = \nu^2$ and we find that~\cref{thm.necessaryLU} does not hold if the factor $\nu^2+\mu^2$ is reduced to $<\nu^2$.
\label{ex:AnotherOne}
\end{example} 

\Cref{ex:FirstOne,ex:AnotherOne} together show that one cannot reduce the factor of $\nu^2+\mu^2$ in~\cref{thm.necessaryLU} to $<\max(\nu^2,\mu^2)$. So, while near-local maximum volume pivoting is both necessary and sufficient for GE, there is a small gap in terms of constants. That is, a partial LU factorization can give singular value estimates of $\sigma_k(A)$ and $\sigma_{k+1}(A)$ to within a factor of $\mu$ as well as interpolation bounds of $\nu$, but only find a near-local maximum submatrix with $\gamma = \max(\nu^2,\mu^2)$.  

As written,~\cref{thm.necessaryLU} is not explicitly assuming that the partial LU factorization is estimating the singular values of $A$ with $A_k$, as the statement involves $A_{11}$. However, we have $\|S(A_{11})\|_2 = \|A-A_k\|_2$ so statement 2 in~\cref{thm.necessaryLU} is an inequality found in~\cref{eq:GoodTrailingSV} with $j=1$.  From~\cref{eq:betweenA11Ak} and the norm equivalence between $\|\cdot\|_2$ and $\|\cdot\|_{\max}$, we also find that
\[
\begin{aligned} 
\sigma_k(A_k)&\leq (1 + \sqrt{k(m-k)}\|A_{21}A_{11}^{-1}\|_{\max})(1 + \sqrt{k(n-k)}\|A_{11}^{-1}A_{12}\|_{\max})\sigma_k(A_{11})\\
&\leq (1 + \sqrt{k(m-k)}\nu)(1 + \sqrt{k(n-k)}\nu)\sigma_k(A_{11}),\\
\end{aligned}
\]
which is an inequality in~\cref{eq:GoodLeadingSV} with $j=k$ after noting that $\sigma_{k}(A_{11})\leq \sigma_{k}(A)$. Hence, a consequence of~\cref{thm.necessaryLU} is the following: 
\begin{corollary}
Let $A\in\mathbb{R}^{m\times n}$ and $1\leq k\leq {\rm rank}(A)$. Any partial LU factorization of $A$ in~\cref{eq.rankkLU} that constructs a rank $k$ approximation $A_k$ such that
\begin{enumerate} 
\item the inequalities~\cref{eq:GoodLeadingSV} and~\cref{eq:GoodTrailingSV} hold for some $\mu_{m,n,k}\geq 1$, and
\item $\|A_{21}A_{11}^{-1}\|_{\max}\leq \nu$ and $\|A_{11}^{-1}A_{12}\|_{\max} \leq \nu$ for some $\nu\geq 1$,
\end{enumerate} 
must have found a near-local maximum volume $k\times k$ submatrix of $A$ with 
\[
\gamma \leq \nu^2 + \mu_{m,n,k}^2(1 + \sqrt{k(m-k)}\nu)(1 + \sqrt{k(n-k)}\nu). 
\]
\label{cor:GEfinal} 
\end{corollary} 

\Cref{cor:GEfinal} makes near-local maximum volume pivoting the archetypal pivoting strategy for rank-revealers based on GE.  It is worth noticing that the proof of~\cref{cor:GEfinal} is not requiring that all the singular value estimates in~\cref{eq:GoodLeadingSV,eq:GoodTrailingSV} hold. The proof only needs two of them! Therefore, if a partial LU factorization computes a rank $k$ approximation $A_k$ with: (i) $\sigma_k(A_k)\approx \sigma_k(A)$, (ii) $\|A-A_k\|_2\approx \sigma_{k+1}(A)$, and (iii) small interpolative bounds, then it must have identified a near-local maximum volume $k\times k$ submatrix of $A$.  Since a near-local maximum volume pivot has been found, the partial LU factorization will be satisfying all the inequalities in~\cref{eq:GoodLeadingSV,eq:GoodTrailingSV} with a relatively small $\mu_{m,n,k}$.

\subsection{All QR rank-revealing factorizations find a near-local maximum volume submatrix}
\label{sec:QR2}
We now consider a partial QR factorization of $A$ given in~\cref{eq.rankkQR}. We show that if (1) $\sigma_k(R_{11})\approx \sigma_k(A)$, (2) $\|R_{22}\|_2\approx \sigma_{k+1}(A)$, and (3) interpolative bounds hold (see~\cref{def:InterpolativeBoundsQR}), then the first $k$ columns of $AP$ are a near-local maximum volume $m\times k$ submatrix of $A$. 

\begin{theorem}\label{thm.necessaryQR2}
Let $A\in\mathbb{R}^{m\times n}$ and $1\leq k\leq {\rm rank}(A)$. Any partial QR factorization of $A$ in~\cref{eq.rankkQR} that satisfies, for some $\mu,\nu\geq 1$,
\begin{enumerate} 
\item $\sigma_k(A) \leq \mu \sigma_k(R_{11})$,
\item $\|R_{22}\|_2\leq \mu\sigma_{k+1}(A)$, and
\item $\|R_{11}^{-1}R_{12}\|_{\max}\leq \nu$,
\end{enumerate} 
must have the first $k$ columns of $AP$ as a near-local maximum volume $m\times k$ submatrix of $A$ with $\gamma\leq \sqrt{\nu^2 + \mu^4}$. 
\end{theorem}
\begin{proof}
Let $B$ be the $m\times k$ submatrix of $A$ formed by taking the first $k$ columns of $AP$. In~\cref{eq:UpdateQR}, there is a formula for ${\rm vol}(\hat{B})/{\rm vol}(B)$, where $\hat{B}$ is obtained from $B$ by interchanging columns $i$ and $k+j$. By the triangle inequality, we have 
\[
\left(\frac{{\rm vol}(\hat{B})}{{\rm vol}(B)}\right)^2 \leq \nu^2 + \| (R_{11}^\top R_{11})^{-1}\|_{\max} \|R_{22}^\top R_{22}\|_{\max}\leq \nu^2 + \frac{\|R_{22}\|_2^2}{\sigma_k^2(R_{11})},
\]
where we have used $\|(R_{11}^\top R_{11})^{-1}\|_{\max}\leq \|(R_{11}^\top R_{11})^{-1}\|_2$, $\|R_{22}^\top R_{22}\|_{\max}\leq \|R_{22}^\top R_{22}\|_2$, $\|R_{22}^\top R_{22}\|_2 = \|R_{22}\|_2^2$, and $\| (R_{11}^\top R_{11})^{-1}\|_{2} = 1/\sigma_k^2(R_{11})$. Since
\[
\frac{\|R_{22}\|_2^2}{\sigma_k^2(R_{11})}\leq \frac{\mu^2 \sigma_{k+1}^2(A)}{\sigma_k^2(A)/\mu^2} = \mu^4\frac{\sigma_{k+1}^2(A)}{\sigma_k^2(A)}\leq \mu^4,
\]
we have $({\rm vol}(\hat{B})/{\rm vol}(B))^2\leq \nu^2+\mu^4$. The result follows as this argument applies to any $\hat{B}$ obtained from $B$ by interchanging columns $i$ and $k+j$. 
\end{proof}

We can easily modify~\cref{thm.necessaryQR2} to be about the constructed rank $k$ approximation of $A$. First, we note that $\|R_{22}\|_2 = \|A-A_k\|_2$ so statement 2 in~\cref{thm.necessaryQR2} is an inequality found in~\cref{eq:GoodTrailingSV} with $j=1$.  Since 
\[
A_k P = Q_1\begin{bmatrix}R_{11} & R_{12} \end{bmatrix} = Q_1R_{11}\begin{bmatrix}I_{k,k} & R_{11}^{-1}R_{12} \end{bmatrix} 
\]
and the norm equivalence between $\|\cdot\|_2$ and $\|\cdot\|_{\max}$, we find that
\[
\begin{aligned} 
\sigma_k(A_k)&\leq (1 + \sqrt{k(n-k)}\|R_{11}^{-1}R_{12}\|_{\max})\sigma_k(R_{11})\leq (1 + \sqrt{k(n-k)}\nu)\sigma_k(R_{11}),\\
\end{aligned}
\]
which is an inequality in~\cref{eq:GoodLeadingSV} with $j=k$ after noting that $\sigma_k(R_{11})\leq \sigma_k(A)$. Hence, a consequence of~\cref{thm.necessaryQR2} is the following: 
\begin{corollary}
Let $A\in\mathbb{R}^{m\times n}$ and $1\leq k\leq {\rm rank}(A)$. Any partial QR factorization of $A$ in~\cref{eq.rankkQR} that constructs a rank $k$ approximation $A_k$ such that
\begin{enumerate} 
\item the inequalities~\cref{eq:GoodLeadingSV} and~\cref{eq:GoodTrailingSV} hold for some $\mu_{m,n,k}\geq 1$, and
\item $\|R_{11}^{-1}R_{12}\|_{\max} \leq \nu$ for some $\nu\geq 1$,
\end{enumerate} 
must have found a near-local maximum volume $m\times k$ submatrix of $A$ with 
\[
\gamma \leq \sqrt{\nu^2 + (1 + \sqrt{k(n-k)}\nu)^2\mu_{m,n,k}^4}. 
\]
\label{cor:QRfinal} 
\end{corollary} 

\Cref{cor:QRfinal} makes near-local maximum volume pivoting the canonical pivoting strategy for rank-revealers based on QR.  The proof of~\cref{cor:QRfinal} does not even need all the singular value estimates in~\cref{eq:GoodLeadingSV,eq:GoodTrailingSV} to be satisfied. In fact, any partial QR factorization that computes a rank $k$ approximation $A_k$ with: (i) $\sigma_k(A_k)\approx \sigma_k(A)$, (ii) $\|A-A_k\|_2\approx \sigma_{k+1}(A)$, and (iii) a small interpolative bound, must have identified a near-local maximum volume $m\times k$ submatrix of $A$.  Since a near-local maximum volume pivot has been found by such a partial QR factorization, the computed $A_k$ will satisfy all the inequalities in~\cref{eq:GoodLeadingSV,eq:GoodTrailingSV} for some $\mu_{m,n,k}$.

\section{Using local maximum volume to assess the success of any pivoting strategy}\label{sec:assessment} 
Since we know the near-local maximum volume is necessary and sufficient for a rank-revealer based on GE and QR, we can develop a metric to assess the success of any pivoting strategy. 

Suppose that a pivoting strategy has selected a submatrix $B$ from $A$. We can measure the success of the pivoting strategy on $A$ by computing the following: 
\begin{equation} 
\mu_B = \max\left\{\max_{\hat{B}\in\mathcal{S}_B} \frac{{\rm vol}(\hat{B})}{{\rm vol}(B)},1\right\},
\label{eq:metric} 
\end{equation} 
where $\mathcal{S}_B$ is the set of all submatrices of $A$ that are neighbors of $B$ in the volume submatrix graph to $B$ (see~\cref{sec:volumegraph}). If $\mu_B=1$ or a small constant, then the pivoting strategy found a local maximum volume pivot and GE and QR are rank-revealers (see~\cref{thm.sufficientLU,thm.sufficientQR}). However,  when $\mu_B$ is large, the pivoting strategy failed to discover a near-local maximum volume pivot and GE and QR must either not be a rank-revealer or must give large interpolative bounds (see~\cref{thm.necessaryLU,thm.necessaryQR2}). In particular, when $\mu_B$ is large, GE and QR will either provide poor singular values estimates of $A$ or the partial factorization will have large interpolative bounds. Therefore, $\mu_B$ in~\cref{eq:metric} can be regarded as a quantity to judge the quality of any pivot. The closer $\mu_B$ is to the value of $1$, the better. In the GitHub repository accompanying the manuscript~\cite{Github}, we provide efficient MATLAB code to compute $\mu_B$ for QR pivots and estimate it for GE pivots. 

After one has computed $\mu_B$, it can immediately use it is to assess the quality of the rank-revealer. For GE, one has 
\[
\frac{1}{1+5\mu_B^2k\sqrt{mn}}\sigma_j(A)\leq \sigma_j(A_k)\leq (1+5\mu_B^2k\sqrt{mn})\sigma_j(A), \qquad 1\leq j\leq k, 
\]
and 
\[
\sigma_{k+j}(A)\leq \sigma_j(A-A_k)\leq (1+5\mu_B^2k\sqrt{mn})\sigma_{k+j}(A), \qquad 1\leq j\leq \min(m,n)-k.
\]
For QR, one has 
\[
\frac{1}{\sqrt{1+5\mu_B^2kn}}\sigma_j(A)\leq \sigma_j(A_k)\leq \sqrt{1+5\mu_B^2kn}\,\sigma_j(A), \qquad 1\leq j\leq k, 
\]
and 
\[
\sigma_{k+j}(A)\leq \sigma_j(A-A_k)\leq \sqrt{1+5\mu_B^2kn}\,\sigma_{k+j}(A), \qquad 1\leq j\leq \min(m,n)-k.
\]
If one finds that $\mu_B$ is large, then the selected pivot is not ideal. However, one can start with that submatrix and run~\cref{alg:NearLocalMaxVol} to improve the pivot.

\subsection{Assessing the success of GECP} 
The metric can be used both analytically and numerically. 
One can analytically bound the metric in~\cref{eq:metric} from below on a family of carefully chosen matrices to show that GECP is not always a rank-revealer. To do this, we need to find just one matrix and one value of $k$ such that $\mu_B$ is extremely large, where $B$ is the $k\times k$ submatrix selected by GECP. Consider the $n\times n$ matrix $A = K_n^\top K_n$ with $k = n-1$, where $K_n$ is the $n\times n$ Kahan matrix~\cite{kahan1966numerical}. Here, $K_n$ is given by 
\begin{equation}\label{eq.kahanmat}
	K_n = \begin{bmatrix}
		1 & -s & -s & \cdots & -s \\
		 & c & -c s & \cdots & -c s \\
		 & & c^2 & \cdots & -c^2 s \\
		 &&& \ddots & \vdots \\
		 &&&& c^{n-1}
	\end{bmatrix}, \qquad c^2 + s^2 = 1, \qquad c, s > 0.
\end{equation}
One can show that in exact arithmetic GECP does no pivoting on $A = K_n^\top K_n$ so it  selects the principal $(n-1)\times (n-1)$ submatrix, $B$, as the pivot. Unfortunately, we have~\cite{kahan1966numerical}
\begin{equation} 
\mu_B\geq \frac{\text{vol}(A(2\!:\!n,2\!:\!n))}{\text{vol}(A(1\!:\!(n\!-\!1),1\!:\!(n\!-\!1)))} = \left(\frac{\text{vol}(K_n(:,2\!:\!n))}{\text{vol}(K_n(:,1\!:\!(n-1)))}\right)^2 \geq  s^2 (1+s)^{2(k-1)}.
\label{eq:BoundonFB}
\end{equation} 
Therefore, $\mu_B$ depends exponentially on $k$ and we conclude that GECP is not a rank-revealer as it failed to select a near-local maximum volume pivot on $A$. 

The metric in~\cref{eq:metric} can also be computed numerically to obtain a measure of the quality of a pivot. It is often the case that GECP is used as a rank-revealing factorization in practical settings. To investigate why, we randomly generate 10,000 $50\times 50$ Gaussian matrices with independent and identically distributed entries. For each matrix, we select a $20\times 20$ pivot using GECP and then compute the value of $\mu_B$ (see~\cref{fig:LUexperimentsmetric} (left)). Astonishingly, we never observe a value of $\mu_B$ greater than 2.

\begin{figure} 
\centering
\begin{minipage}{.49\textwidth}
\begin{overpic}[width=\textwidth]{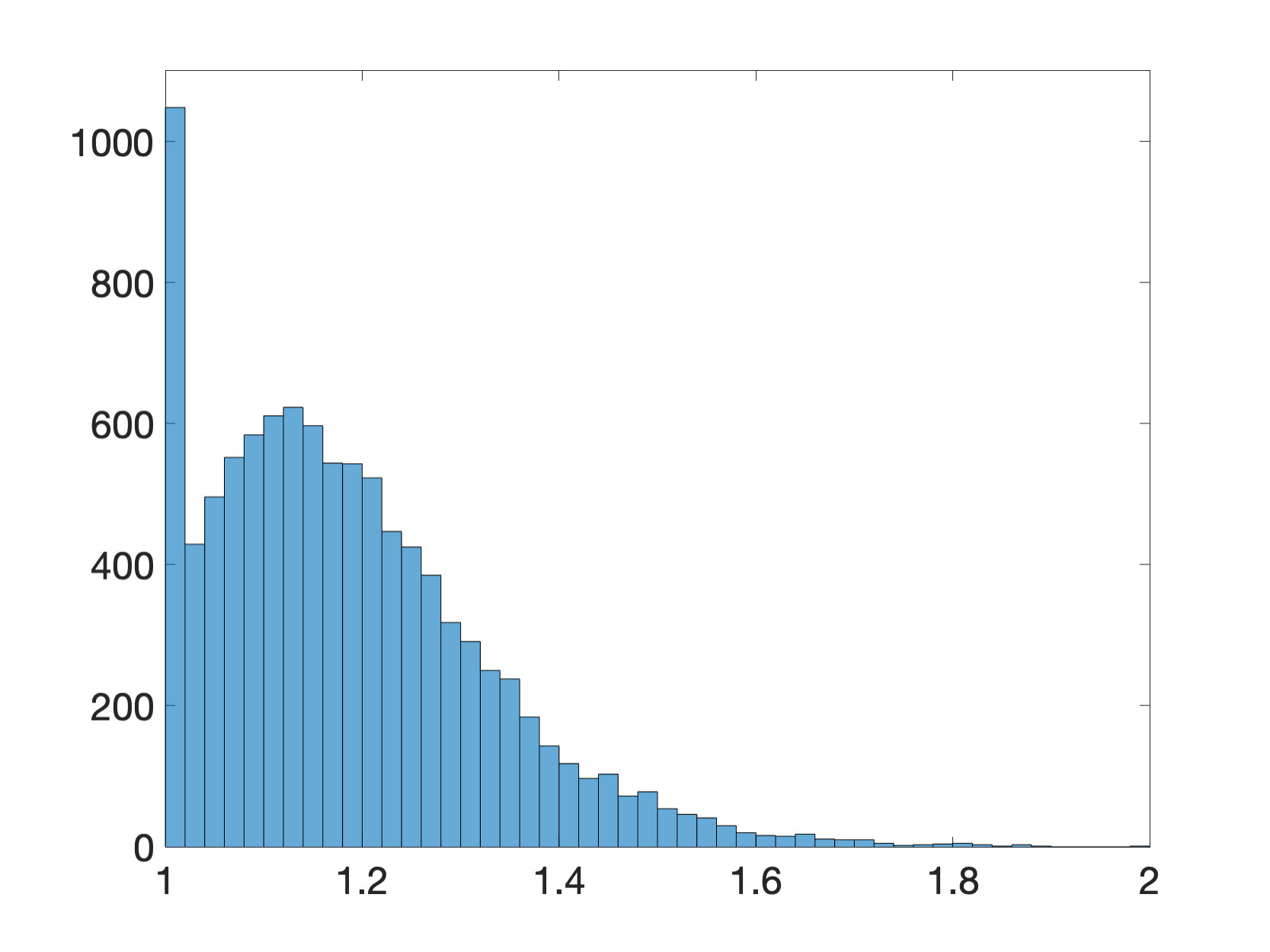}
\put(48,0) {$\mu_B$} 
\put(43,72) {GECP} 
\end{overpic} 
\end{minipage}
\begin{minipage}{.49\textwidth}
\begin{overpic}[width=\textwidth]{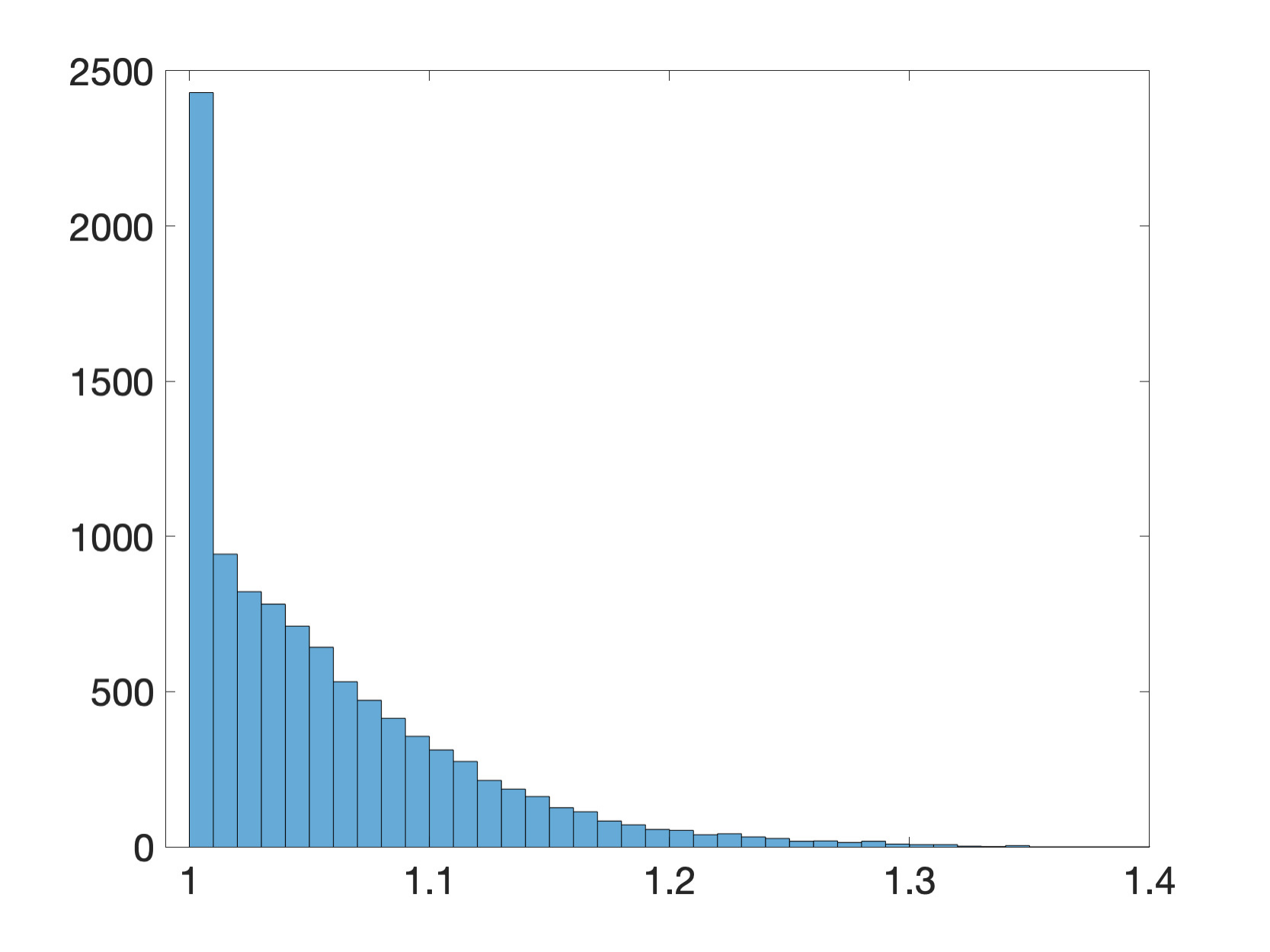}
\put(48,0) {$\mu_B$} 
\put(43,72) {CPQR} 
\end{overpic} 
\end{minipage}
\caption{We observe that GECP (left) and CPQR (right) find near-local maximum volume submatrices with small $\mu_B$ on $50\times 50$ random Gaussian matrices with extremely high probability and $k=20$. Here, we randomly generate 10,000 Gaussian matrices, compute the metric in~\cref{eq:metric}, and plot a histogram. Despite the worse-case bound on $\mu_B$ being exponential in $k$ for GECP and CPQR, most of the time these pivoting strategies find submatrices with small $\mu_B$.}
\label{fig:LUexperimentsmetric} 
\end{figure}

\subsection{Assessing the success of CPQR}\label{sec:CPQRmetric} 
CPQR is not a rank-revealer because it can select a submatrix $B$ with a large $\mu_B$. To see this, consider the $n\times n$ matrix $A = K_n$ with $k=n-1$, where $K_n$ is the Kahan matrix in~\cref{eq.kahanmat}. It can be observed that CPQR does not pivot on $K_n$ so CPQR selects the first $n-1$ columns of $A$. However, from~\cref{eq:BoundonFB}, we see that 
\[
\frac{\text{vol}(K_n(:,2\!:\!n))}{\text{vol}(K_n(:,1\!:\!(n-1)))} \geq  s (1+s)^{k-1},
\]
which is exponential in $k$. 
 
Nevertheless, CPQR is often used in practice as a rank-revealer. To see why, we randomly generate 10,000 $50\times 50$ Gaussian matrices with independent and identically distributed entries. For each matrix, we select a $50\times 20$ pivot using CPQR and then compute the value of $\mu_B$ (see~\cref{fig:LUexperimentsmetric} (right)). Astonishingly, we never observe a value of $\mu_B$ greater than $\sqrt{2}$.

\begin{remark}
We note that these ``worst case'' examples (both for CPQR and GECP) are quite delicate. In particular, because $\sigma_{n-1}(K_n) \geq \sigma_{n-1}(R_{11}) \geq \sigma_{n}(K_n)$ it is only possible to have an exponential gap between $\sigma_{n-1}(R_{11})$ and $\sigma_{n-1}(K_n)$ if $K_n$ itself has an exponential gap between $\sigma_{n-1}(K_n)$ and $\sigma_n(K_n)$. Such a gap is not stable under perturbation.
\end{remark}
 
\section{Applications}\label{sec:applications}
It is useful to see some of the applications of rank-revealers. We look at: (i) Pseudoskeleton approximation of kernels, (ii) Model order reduction, and (iii) Quantum chemistry.  Throughout this section, we will use notation that is more familiar to domain researchers.  

\subsection{Pseudoskeleton approximation of kernels}\label{sec:FunctionApproximation} 
Pseudoskeleton approximation is a technique used to approximate large and typically dense matrices, which often arise in kernel methods~\cite{goreinov1997theory}. It selects a subset of the columns and rows from the original matrix---informally known as a skeleton---to form a low-rank approximation. The pseudoskeleton approximation is equivalent to the low-rank approximation formed by GE, and the task is to find a reasonable subset of the columns and rows. It is common to use either GE with rook pivoting\footnote{GE with rook pivoting selects pivots one at a time so that each pivot is the largest entry in absolute value in both its column and row. GE with rook pivoting is not a rank-revealer as it can fail to find a near-local maximum volume pivot on diagonal matrices.} or GECP to select the columns and rows. The link between pseudoskeleton approximation and GE with local maximum volume pivoting is first mentioned by Bebendorf in 2000~\cite{bebendorf2000approximation}. 

Pseudoskeleton approximation is used for hierarchical low-rank matrix compression~\cite{bebendorf2000approximation,ho2012recursive,minden2017recursive} and function approximation~\cite{townsend2013extension}. In~\cref{fig:SkeletonAiry} we show the skeleton grids selected by GECP for two functions. It is common for the pivots to end up in areas of the function where there is more variation. 

\begin{figure} 
\centering
\begin{minipage}{.49\textwidth} 
\centering
\begin{overpic}[width=.8\textwidth,trim={130 75 130 50},clip]{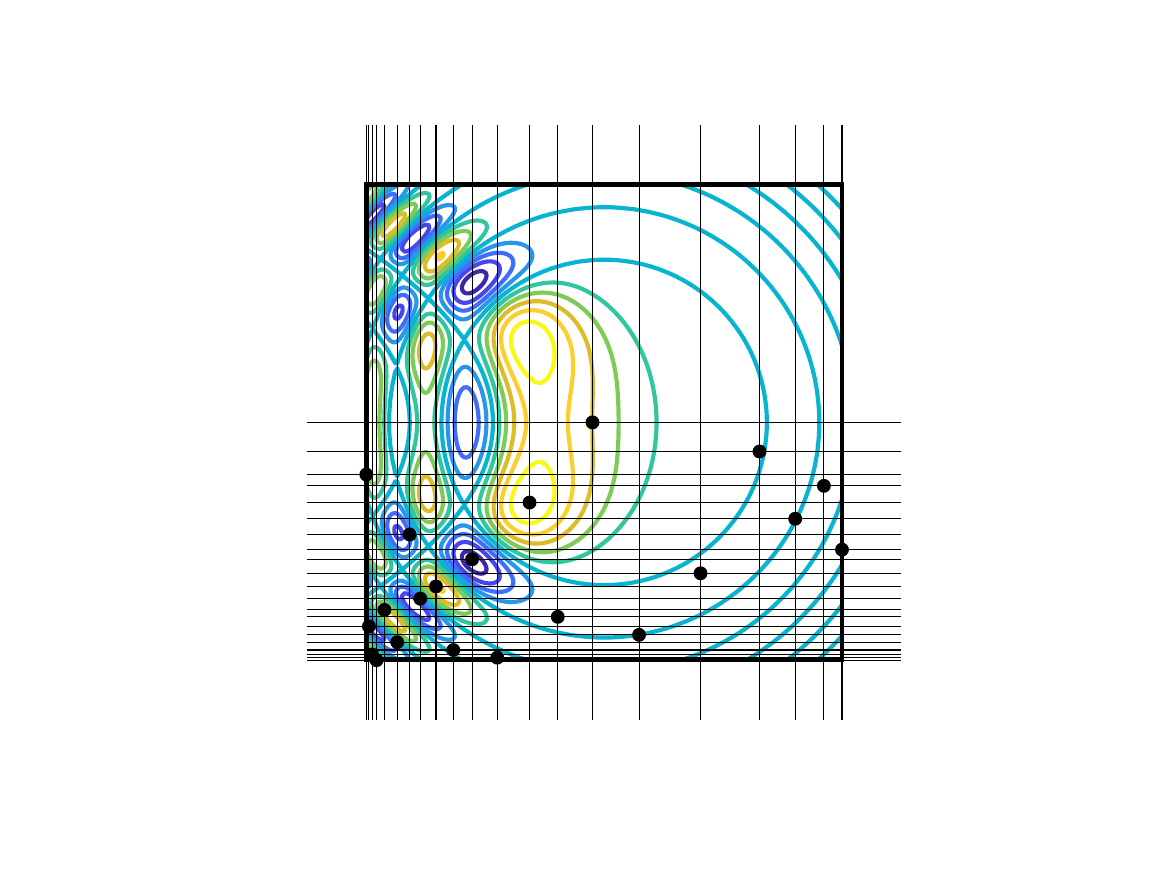}
\end{overpic} 
\end{minipage} 
\begin{minipage}{.49\textwidth} 
\centering
\begin{overpic}[width=.8\textwidth,trim={130 75 130 50},clip]{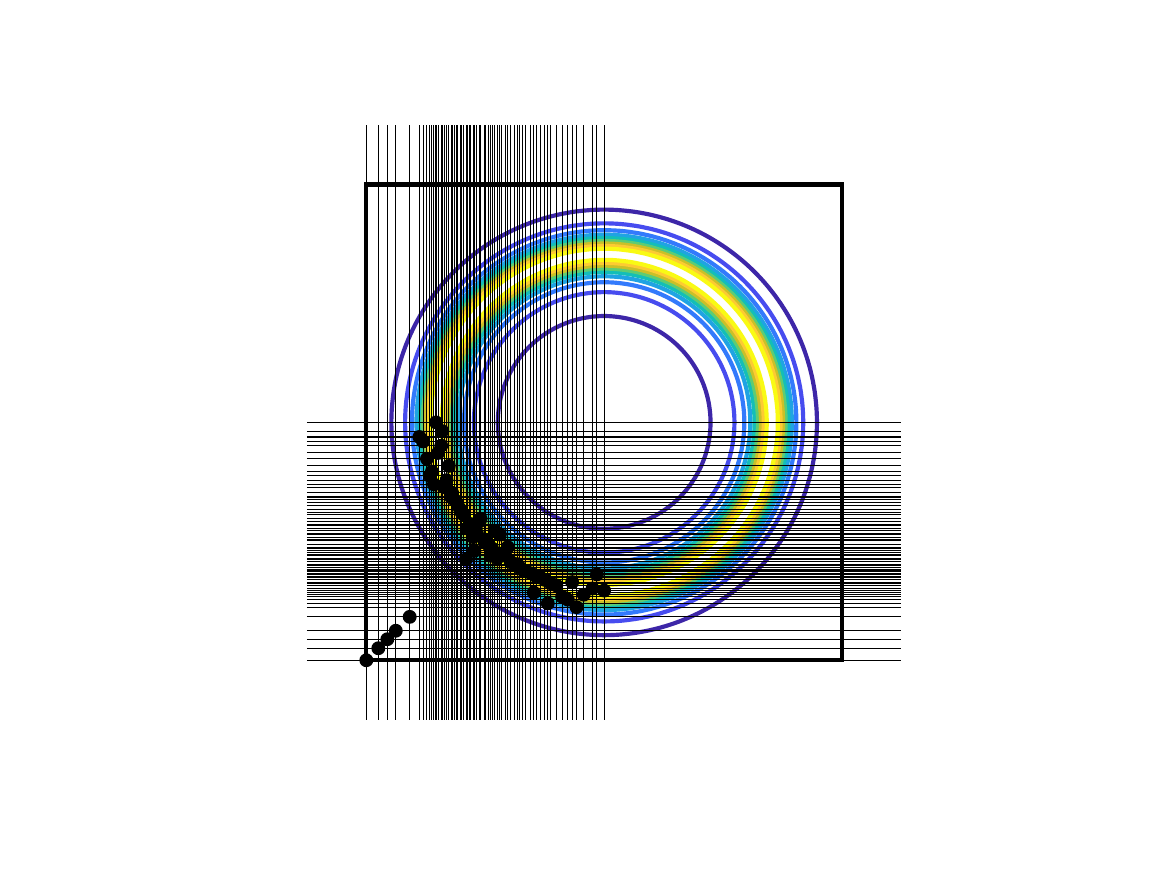}
\end{overpic} 
\end{minipage} 
\caption{Left: The skeleton selected by GECP for the function $f(x,y) = {\rm Ai}(5(x+y^2)){\rm Ai}(-5(x^2+y))$ 
on $[-1,1]\times [-1,1]$ when discretized on a $129\times 129$ bivariate Chebyshev grid and $k=20$. Right: GECP skeleton for the function $f(x,y) = 1/(1+100(1/2-x^2-y^2)^2)$ on $[-1,1]\times [-1,1]$ when discretized on a $513\times 513$ bivariate Chebyshev grid and $k=65$.}
\label{fig:SkeletonAiry} 
\end{figure} 

The reason that GECP is a popular algorithm in pseudoskeleton approximation world is because it is often observed to deliver good singular value estimates and satisfy interpolative bounds. To show this, consider the 2D Runge-like functions given by 
\begin{equation} 
f_\beta(x,y) = \frac{1}{1 + \beta (x^2 + y^2)^2}, \qquad \beta >0,
\label{eq:RungeLikeFunctions} 
\end{equation} 
which is analytic in both variables. \Cref{fig:Skeleton} (left) shows that when $k=5$ the low-rank approximation $A_k$ formed by GECP provides excellent singular value estimates for the leading $k$ singular values. Moreover, $A-A_k$ provides excellent singular value estimates for the trailing singular values.  

We also consider Wendland's compactly-supported radial basis functions~\cite{wendland1995piecewise}
\begin{equation} 
\phi_{3,s}(x,y) = \begin{cases} (1-d)_+^2,& s=0,\\ (1-d)_+^4(4d+1), & s=1,\\ (1-d)_+^8(32d^3+25d^2+8d+1),& s=3,\\\end{cases}
\label{eq:WendlandFunctions} 
\end{equation} 
where $(x)_+$ is the function that is equal to $x$ when $x\geq 0$ and is $0$ when $x<0$ and $d = |x-y|$ is the distance between $x$ and $y$. The function $\phi_{3,s}$ has $2s$ continuous derivatives in both variables. These Wendland functions are discretized on a grid to form a kernel matrix. The smoothness of the functions can be used to explain the observed decay rates of the singular values~\cite[Sec.~3.4]{townsend2014computing}. 
\begin{figure} 
\centering
\begin{minipage}{.49\textwidth} 
\begin{overpic}[width=\textwidth]{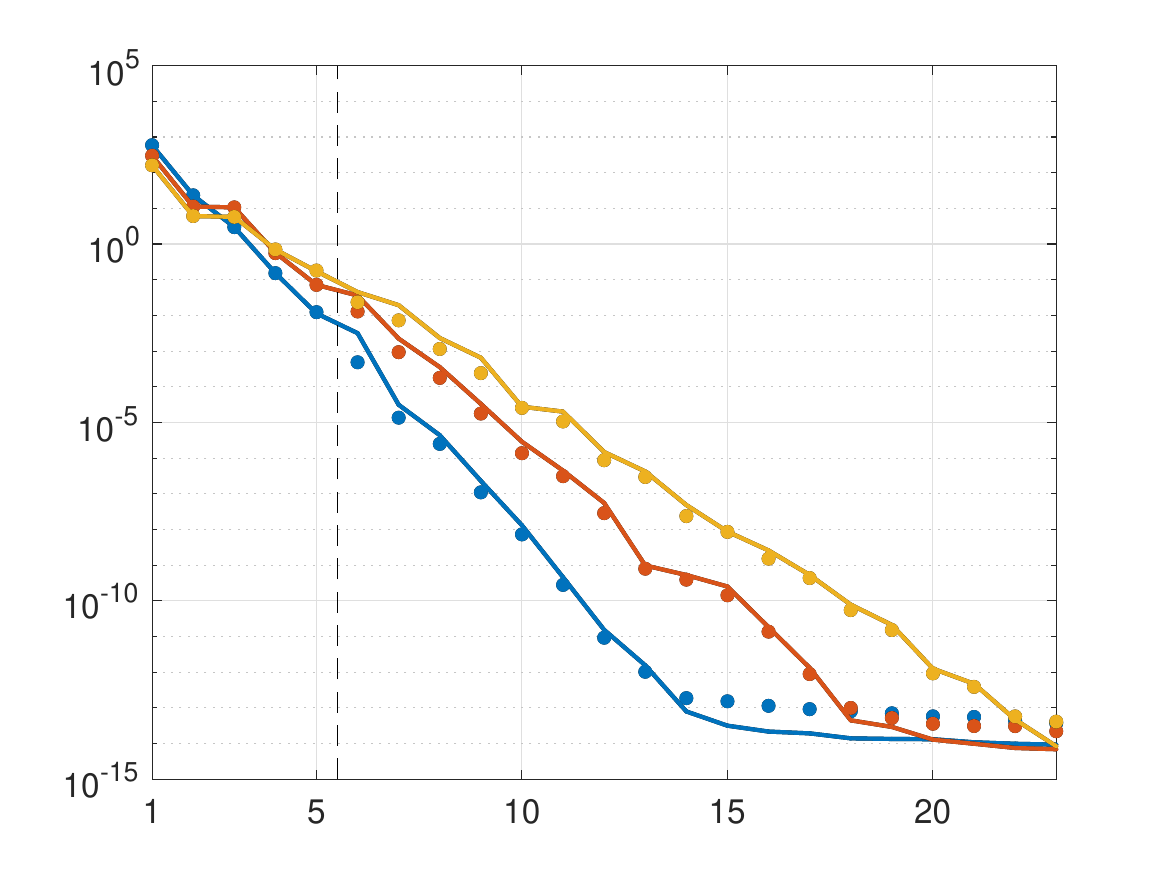}
\put(50,0) {$j$}
\put(18,10) {\rotatebox{90}{$\sigma_j(A)\approx \sigma_j(A_k)$}}
\put(33,58) {$\sigma_j(A-A_k)\approx \sigma_{j+k}(A)$}
\put(67,30) {\rotatebox{-38}{$\beta =100$}}
\put(65,23) {\rotatebox{-45}{$\beta =10$}}
\put(50,24) {\rotatebox{-46}{$\beta =1$}}
\end{overpic} 
\end{minipage} 
\begin{minipage}{.49\textwidth} 
\begin{overpic}[width=\textwidth]{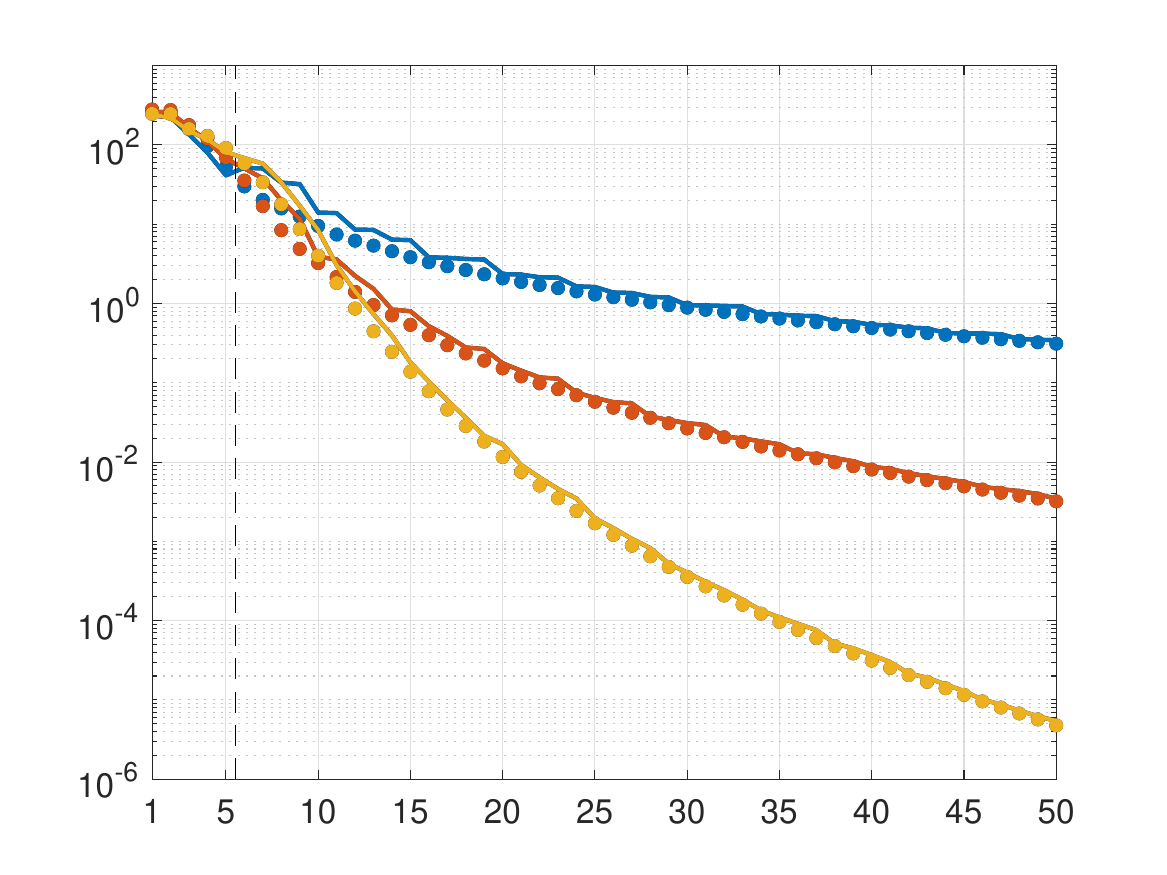}
\put(50,0) {$j$}
\put(14,10) {\rotatebox{90}{$\sigma_j(A)\approx \sigma_j(A_k)$}}
\put(33,58) {$\sigma_j(A-A_k)\approx \sigma_{j+k}(A)$}
\put(67,50) {\rotatebox{-5}{$\phi_{3,0}$}}
\put(67,38) {\rotatebox{-7}{$\phi_{3,1}$}}
\put(67,23) {\rotatebox{-15}{$\phi_{3,3}$}}
\end{overpic} 
\end{minipage} 
\caption{Using GECP to estimate the singular values of a $1000\times 1000$ kernel matrix formed by discretizing functions on $(x,y)\in [-1,1]^2$ with a bivariate Chebyshev grid. The kernel matrices are formed by discretizing the Runge-like functions in~\cref{eq:RungeLikeFunctions} (left) with $\beta = 1$ (blue), $\beta=10$ (red), and $\beta = 100$ (yellow) and Wendland's RBFs in~\cref{eq:WendlandFunctions}. Here, solid dots are the exact singular values of the discretized kernel matrix while the solid lines interpolate the GECP's singular values estimates.}
\label{fig:Skeleton} 
\end{figure}
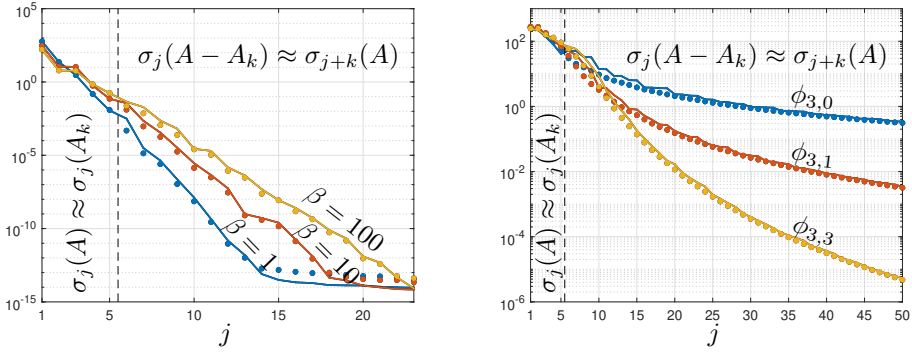 

When $k = 5$, we find that GECP finds a near-local maximum volume pivot for all the functions in~\cref{eq:RungeLikeFunctions} with $\beta = 1,10,100$ and~\cref{eq:WendlandFunctions} when they are discretized on a $1000\times 1000$ bivariate Chebyshev grid. This appears to be an extremely common feature of GECP for pseudospectral approximation and for these six  examples we have $\mu_B\leq 2$. The two-dimensional version of Chebfun relies on GECP as a rank-revealer for bivariate function approximation~\cite{townsend2013extension}.  If one wants theoretical statements, then one can always use the $k\times k$ pivot selected by GECP as an initial submatrix for~\cref{alg:NearLocalMaxVol}. Then, the quality of the singular value estimates are guaranteed (see~\cref{thm.sufficientLUnear}). 

\subsection{Model order reduction}\label{sec:ROMapplication} 
Model order reduction (MOR) is a well-known approach to lower the computational complexity of expensive high-dimensional (possibly parametric) problems. The goal of MOR is to derive a low-dimensional reduced-order model, which can be evaluated significantly faster than the high-dimensional model. One popular method to derive a reduced-order model is using Galerkin projection combined with the proper orthogonal decomposition (POD), approximating the high-dimensional state by $\state \approx \rob \statered$ with $\state \in \R^{n}$, reduced state $\statered \in \R^{r}$, and $\rob \in \R^{n \times r}$ being the reduced-order basis, possibly generated by the POD.  
However, in case that a nonlinear term $f: \R^n \rightarrow \R^{n}$ is present in the problem formulation, even though $\statered$ lives in the low-dimensional space, the evaluation of the respective reduced term $f_r(\statered) = \rob^\top f(\rob \statered)$ still scales with the high dimension $n$, degrading the performance of the ROM. 
To overcome this issue, the discrete empirical interpolation method (DEIM) has been introduced in \cite{chaturantabut2010nonlinear} as a discrete version of the EIM \cite{barrault2004empirical}. 
Given a matrix $\bm{U}\in \R^{n \times m}$ of rank $m$ and a nonlinear function $f: D \rightarrow \R^{n}$, $D \subset \R^d$ the DEIM approximation of $f$ is for $\tau \in D$ defined by 
\begin{align*}
	\hat{f}(\tau) := \bm{U}(\mathbb{S}^\top\bm{U})^{-1} \mathbb{S}^\top f(\tau)
\end{align*}
where $\mathbb{S}\in \R^{n \times m}$ is the so-called selection operator, consisting of columns of the $n \times n$ identity matrix. Then $f_r(\statered) \approx \rob^\top \bm{U}(\mathbb{S}^\top \bm{U})^{-1} \mathbb{S}^\top f(\rob \statered)$ can be computed without lifting $\statered$ to the high-dimensional space. 

For an orthonormal $\bm{U} \in \R^{n \times m}$, $m \le n$, the following error bound can be derived, see \cite[Lemma 3.2]{chaturantabut2010nonlinear} 
\begin{align*}
	\| f - \hat{f} \|_2 \le \|(\mathbb{S}^\top\bm{U})^{-1} \|_2 \| f - \bm{U} \bm{U}^\top f\|_2. 
\end{align*}
This means that in this application, we have row subset selection problem, where the goal (when using QR) is to minimize 
\begin{align*}
	\|(\mathbb{S}^\top\bm{U})^{-1} \|_2 = \| R_{11}^{-1}\|_2
\end{align*}
i.e., the bound/approximation quality does depend on $\mu_{m,n,k}$ in~\cref{eq:GoodLeadingSV}.\footnote{In this section, we consider $k=m$.}
In \cite{drmac2016new} the QDEIM method has been introduced, which presented a new selection operator based on CPQR. It can be shown that this method has the error bound
\begin{align*}
 	\|(\mathbb{S}^\top \bm{U})^{-1} \|_2 \le \sqrt{n-m+1} \frac{\sqrt{4^m+6m-1}}{3},
\end{align*} 
which is know to overestimate the error in practice. Additionally, the authors in \cite{drmac2016new} showed, following an argument from \cite{goreinov1997theory}, that there exists a choice of selection operator such that $\|(\mathbb{S}^\top \bm{U})^{-1} \|_2 \le \sqrt{1 + m(n-m)}$, choosing the global $m \times m$ maximum volume subset, which is computationally infeasible. 
In recent years, the QDEIM method has been extended to a strong rank-revealing QDEIM in \cite{drmac2018discrete} using the rank-revealing algorithm of \cite{gu1996efficient}, leading to 
\begin{align} \label{eqn:srrqr_bound_qdeim}
 	\|(\mathbb{S}^\top\bm{U})^{-1} \|_2 \le \sqrt{1 + \eta^2 m(n-m)}, 
\end{align} 
where $\eta \ge 1$ is a tuning parameter introduced in \cite{gu1996efficient}. 

After considering the theoretical bounds, we investigate how well CPQR performs in practice and how far away it is from QR with local maximum volume, which yields significantly better bounds. To this end, we investigate an illustrative example similar to \cite{quarteroni2015reduced,drmac2018discrete,saibaba2020randomized}. The problem is about model reduction of a diffusion-convection-reaction equation with a nonlinear source term, which, if not approximated with DEIM/QDEIM degrades the performance of the reduced-order model. We set the spatial domain $\Omega= [0,1]^2$, The non-affine source term is given by 
\begin{align*}
s(\state; \bm{\mu}) := \exp\left( - \frac{(x_1 - \mu_1)^2 + (x_2 - \mu_2)^2}{\mu_3^2} \right). 
\end{align*}
with $\mu_1 \in [0.2,0.8]$, $\mu_2 \in [0.15,0.35]$ and $\mu_3 \in [0.1, 0.35]$. We run QDEIM using a uniform spatial discretization with $100$ points in each direction, a Latin hypercube sampling with $n_s =1000$ and a POD-tolerance of $10^{-6}$. The error is computed by averaging over $200$ test points. We evaluate the QDEIM approximation with CPQR for $k=2,\ldots,30$.

Using the QDEIM with CPQR in this example, the largest value of $\mu_{\textnormal{QDEIM}}= 1.1448$ was obtained for dimension $19$. We can conclude that  
\[
\|(\mathbb{S}^\top\bm{U})^{-1} \|_2 \le \sqrt{1 + 1.1448^2 m(n-m)}.
\]
Moreover, for all dimensions, it took at most $11$, on average $2.3448$ swaps to arrive at a local maximum volume subset, see \cref{fig:QDEIM} (left). In \cref{fig:QDEIM} (right), the error of the QDEIM is plotted with respect to the CPQR pivoting and the final local maximum volume subset configuration. Notably, the QDEIM with CPQR even ended up with a local maximum volume subset in $7$ cases. This shows that in practice QDEIM with CPQR does perform well. With either a check of what the value of $\mu_B$ is in \eqref{eq:metric}, or performing a couple more swaps to arrive at a local maximum volume subset, one can ensure stronger theoretical bounds.  

\begin{figure}[ht!] 
\centering
\begin{minipage}{.49\textwidth}
\centering
\begin{overpic}[width=.9\textwidth,trim=112 70 100 50,clip]{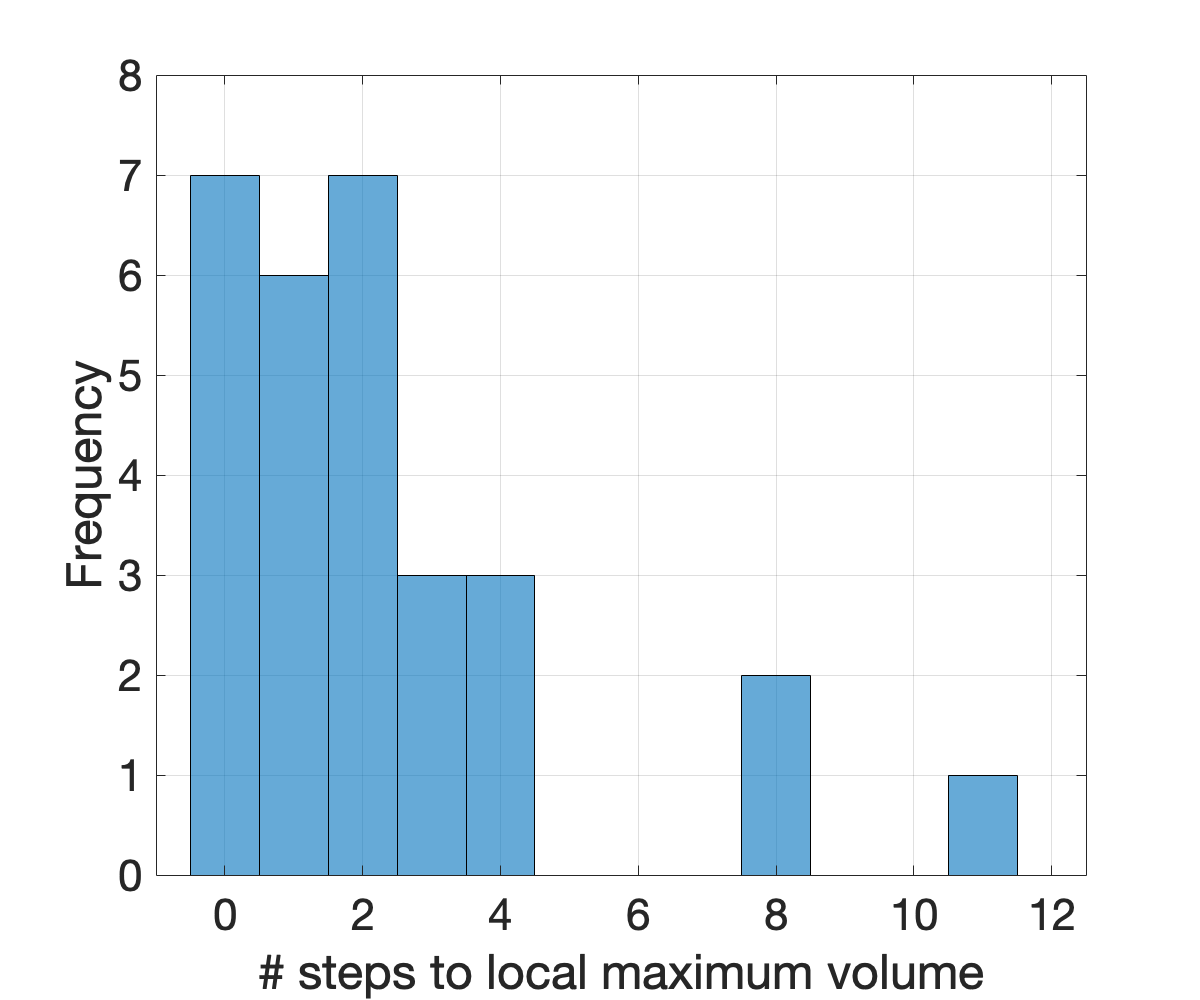} 
\put(-6,30) {\rotatebox{90}{Frequency}}
\put(4,-5) {\# swaps to local maxvol subset}
\end{overpic} 
\end{minipage}
\begin{minipage}{.49\textwidth}
\centering
\begin{overpic}[width=.9\textwidth,trim=112 70 90 50,clip]{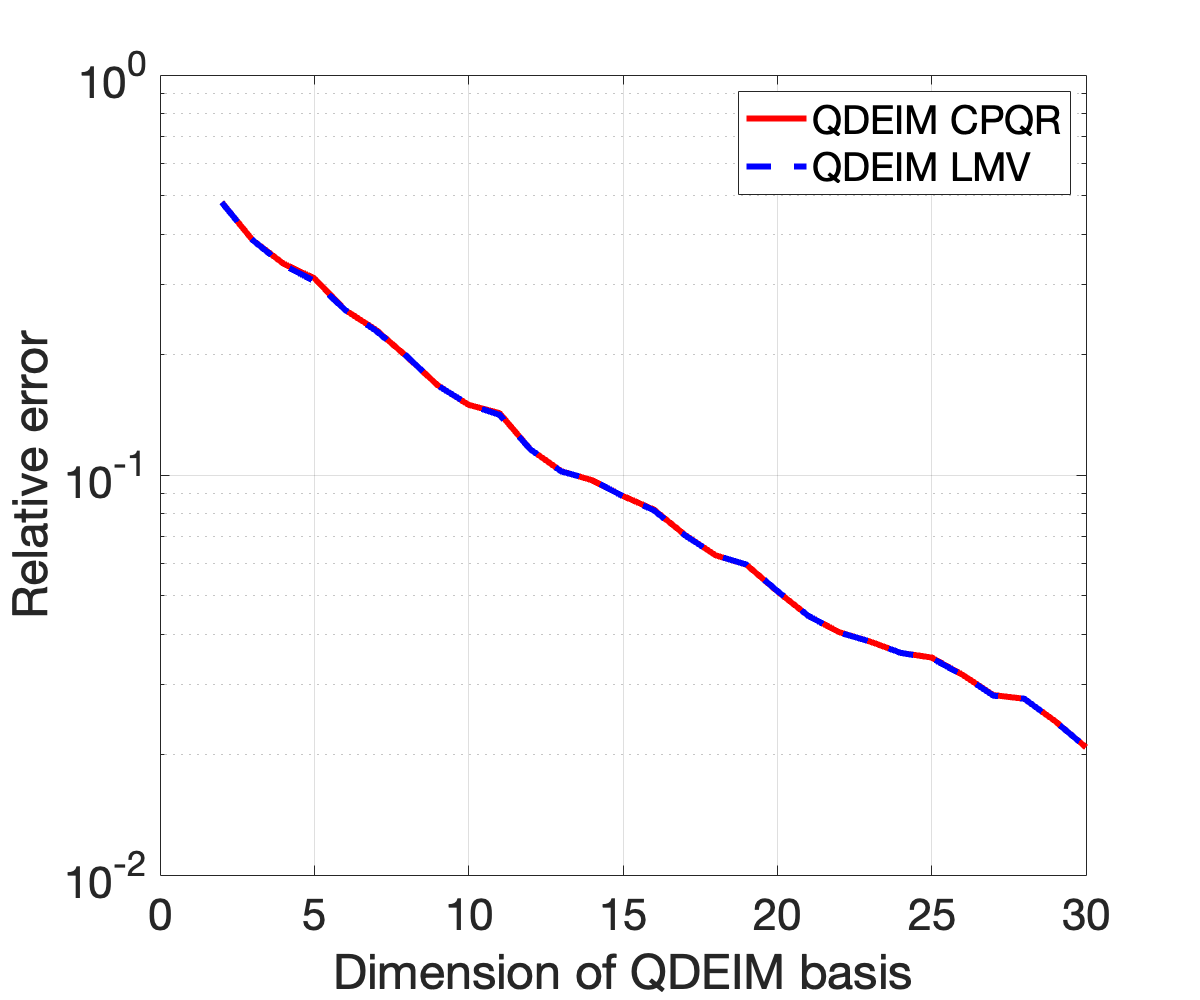} 
\put(-6,30) {\rotatebox{90}{Relative error}}
\put(15,-5) {Dimension of QDEIM basis}
\end{overpic} 
\end{minipage}
\caption{Left: Number of swaps needed to arrive at local maximum volume subset, starting at the subset selected by CPQR. Right: Relative error in $\| \cdot \|_2$ of QDEIM with CPQR (QDEIM CPQR) and the error of QDEIM with respective local maximum volume subset (QDEIM LMV), both with respect to $\| \cdot \|_2$.}
\label{fig:QDEIM} 
\end{figure}

\subsection{Computing spatially localized basis functions}\label{sec:QCapplication}
We now consider an application relevant algorithm whose efficacy is determined by~\cref{eq:GoodLeadingSV}---specifically how large $\sigma_k(R_{11})$ is---when using pivoted QR as a rank-revealer. Specifically, we consider a problem from electronic structure theory, where we seek to understand the properties of molecules and solids from first principles calculations. One common model within this setting is Kohn--Sham density function theory (KSDFT) \cite{HohenbergKohn1964,KohnSham1965}. Within this framework, understanding ground state properties of the underlying atomic system requires finding and working with eigenfunctions $\psi_i$ associated with the $N$ algebraically smallest eigenvalues of the Kohn--Sham Hamiltonian (a non-linear eigenvalue/vector problem); see, e.g.,~\cite{lin2019numerical} for details from a computational perspective.  

Interestingly, many properties of atomic systems (and computational schemes) depend only on the span of $\{\psi_i\}_{i=1}^N$ and not the individual $\psi_i.$ Accordingly, it is often desirable to work with an alternative orthonormal basis $\{\phi_i\}_{i=1}^N$ where each basis function $\phi_i$ is highly localized spatially (see~\cref{fig:alkane} for an example). This is referred to as computing localized Wannier functions (LWFs)~\cite{WannierReview,MarzariVanderbilt1997}.\footnote{Notably, the $\psi_i$ do not naturally have this property, and for simplicity, we only consider insulating systems (i.e., those where there is a gap between $\lambda_N$ and $\lambda_{N+1}$), where such localized orbitals provably exist~\cite{BenziBoitoRazouk2013,Nenciu1983,brouder2007exponential,Kohn1959}.} Recent work connects the computation of LWFs to rank-revealer's based on QR~\cite{damle2018disentanglement,damle2017scdm,damle2015compressed}. 

We will refer to a discretized set of wavefunctions $\psi_i(\br)$ using the matrix $\Psi = \begin{bmatrix}\psi_1(\br) & \cdots & \psi_{N}(\br)\end{bmatrix}.$ Mathematically, the goal is then to compute an orthogonal matrix $U$ such that the columns of $\Phi = \Psi U$ are well-localized (i.e., with respect to the domain of the underlying atomic system). The selected columns of the density matrix (SCDM) algorithm~\cite{damle2018disentanglement,damle2017scdm,damle2015compressed} accomplishes this task by leveraging the fact that if a local basis exists for the span of $\Psi$ then the spectral projector $\Psi\Psi^\top$ must have decay properties~\cite{BenziBoitoRazouk2013}. Therefore, the SCDM algorithm uses a partial QR factorization 
\[
\Psi^\top \Pi = Q \begin{bmatrix} R_{11} & R_{12} \end{bmatrix}
\] 
to identify $N$ exemplar columns of $\Psi\Psi^\top$ to use as candidates for LWFs. Letting $\cC$ denote the first $N$ indices selected by the permutation $\Pi$ we then solve for $U$ via the Orthogonal Procrustes problem $U = \argmin_Z \|\Psi\Psi(\cC,:)^\top - \Psi Z\|_F$. The locality of the basis computed by SCDM is dependent on $\sigma_k(R_{11}).$ Broadly speaking, closer it is to 1 the more localized the resulting basis is. Importantly, because we are using the rank-revealing algorithm to select columns and not build an approximation of $\Psi^\top$,~\cref{eq:GoodTrailingSV} is irrelevant and we only care about~\eqref{eq:GoodLeadingSV} and~\cref{def:InterpolativeBoundsQR}.

\begin{figure}[ht!] 
\centering
\begin{minipage}{.3\textwidth}
\includegraphics[width=\textwidth]{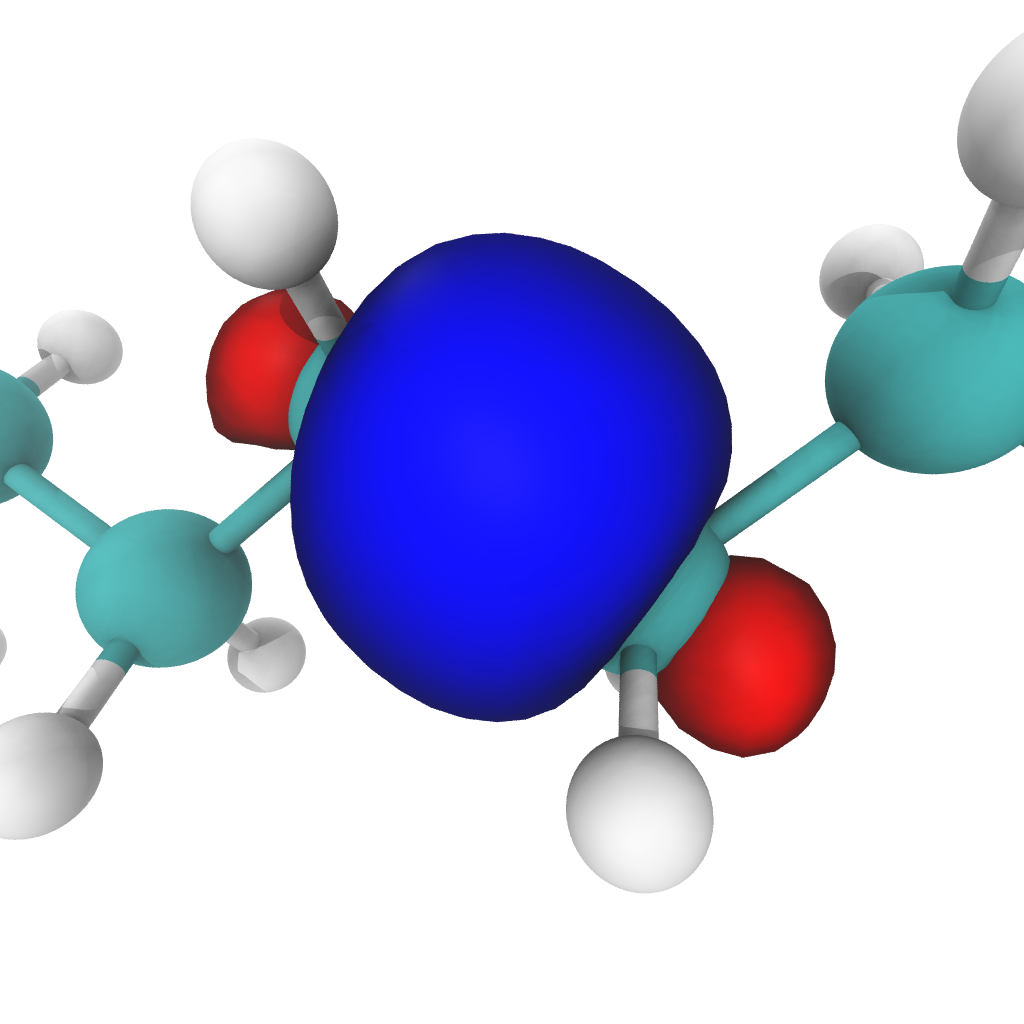} 
\end{minipage}
\hspace{3em}
\begin{minipage}{.3\textwidth}
\includegraphics[width=\textwidth]{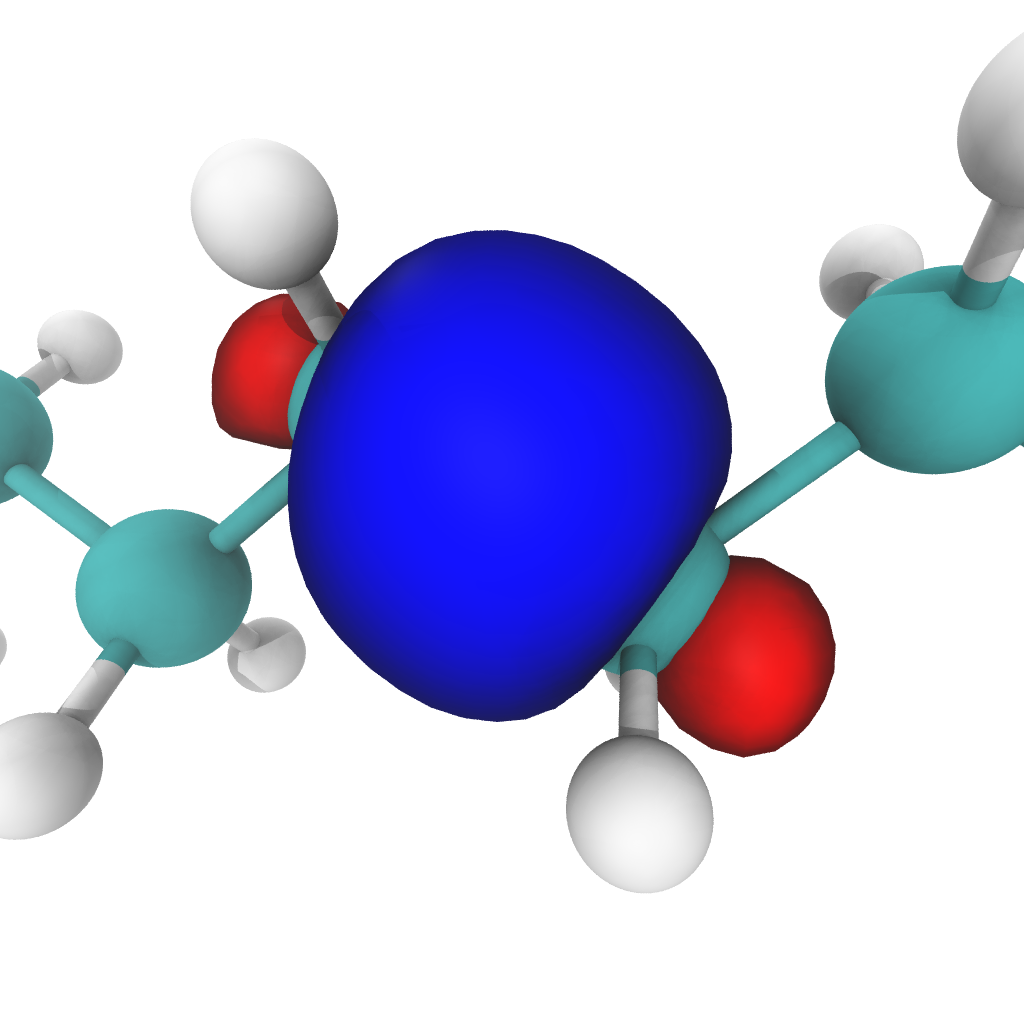} 
\end{minipage}
\caption{One LWF for an alkane system computed via the SCDM methodology using a CPQR (left) and local maximum volume subset (right). The spread of the function produced by the subset that is of locally maximum volume is $0.5\%$ smaller.}
\label{fig:alkane} 
\end{figure}

To illustrate both the efficacy of CPQR within the SCDM framework and the short paths to a subset of locally maximum volume we consider two example atomic systems. Our first example is an alkane system where $n=100$ and the wavefunctions are discretized on 820,125 points, i.e., $\Psi\in\R^{100\times 820,125}$. For this problem, using a CPQR factorization yielded $\mu_\Psi = 1.0012$---providing a subset strikingly close to one of locally maximum volume property. Moreover, it only took 10 swaps to find a subset with $\mu_\Psi = 1.$ LWFs computed using the subset returned by CPQR and the subset of locally maximum volume are nearly visually indistinguishable (see~\cref{fig:alkane}) and the spread is within $0.5\%.$ The subset of locally maximum volume did produce slightly more localized functions---albeit not in a way that would meaningfully impact downstream applications. For a water system with $N=256$ and discretized using 1,953,125 grid points. CPQR yields a subset with $\mu_\Psi = 1.0211$ and it took 9 swaps to find a subset with $\mu_\Psi = 1.$ The LWFs are visually indistinguishable and differ in spread by only $0.02\%$ (again with the subset of locally maximum volume yielding a slightly lower value). 

\begin{remark}
This example also shows the potential importance of a good initialization when seeking subset that has the locally maximum volume property. While it took less than 10 swaps to find a subset with $\mu_\Psi = 1$ when starting from a CPQR selected subset, starting from a random subset took $>100$ swaps (at which point the search was terminated) for the alkane example.
\end{remark}

\section{Discussion}
By taking the perspective that rank-revealers are algorithms and not factorizations, we provided a unified perspective on estimating a matrices singular values using structured low-rank approximations (see~\cref{eq:GoodLeadingSV,eq:GoodTrailingSV}). When specialized to GE and QR style factorizations, we show that finding a submatrix of (nearly) locally maximum volume is both necessary and sufficient for an algorithm to be a rank-revealer. This result is striking on two counts: (1) it shows that any local minimizer on the graph of appropriately sized subsets is good---even if it is far from the global optimizer in terms of volume---and (2) the archetypal algorithm for rank-revealing is to search for a subset that has (nearly) the local maximum volume property. 

Many applications also require that the computed factorization satisfies an interpolation bound (see~\cref{def:InterpolativeBoundsQR,def:InterpolativeBoundsGE}). Analogous to the results about singular values, locally maximum volume subsets satisfy this property and are necessary to do so. Complementing our theoretical results, numerous examples show that many of our bounds cannot be meaningfully improved, highlighting the importance of the local maximum volume concept. This concept also provides clarity on the (vast) prior literature about rank-revealers by showing what are essential parts of various definitions. Lastly, using both random matrices and application derived ones we show that classical algorithms such as GECP and CPQR provide effective initialization when searching for subset of (nearly) locally maximum volume; illustrating the practicality of such a scheme.

\appendix

\section{On the sharpness of~\cref{thm.sufficientLU}}\label{sec:AppendixSharpness}
In~\cref{ex:sharpness} we give an $m\times n$ matrix $A$ that depends on $k$ that demonstrates that the value of $\mu_{m,n,k}$ in~\cref{thm.sufficientLU} must be greater than or equal to $(k+1)\sqrt{(m-k)(n-k)}$. Here, we give the details for that example. Recall that in this example we have $k\geq 2$. 

First, we want to check that the principal $k\times k$ submatrix $A_{11}$ of $A$ is a local maximum volume submatrix. Due to the form of $A_{11}$, it can be shown that  
\[
A_{11}^{-1} = \frac{1}{2(k+2)}\begin{bmatrix}3 & 1 & \cdots &1 & 1 \\ 1 & 3 & \cdots & 1& 1 \\ \vdots & \vdots & \ddots & \vdots & \vdots\\ 1 & 1 & \cdots & 3 & 1\\ 1 & 1 & \cdots & 1 & 3 \end{bmatrix}. 
\]

We need to check that ${\rm vol}(A_{11}) \geq {\rm vol}(\hat{A}_{11})$, where $\hat{A}_{11}$ is a submatrix of $A$ that differs from $A_{11}$ in at most one column and row.  Looking at $A$ and the columns and rows that can be swapped in, we find that the $A_{11}$ and $\hat{A}_{11}$ can differ in at most $2$ diagonal entries. If they differ in two diagonal entries, then ${\rm vol}(\hat{A}_{11})=0$ so we can focus on the situation where they differ in a single entry. Since $A -\hat{A}_{11}$ is a rank one matrix, the matrix determinant lemma~\cite[Lem.~1.1]{ding2007eigenvalues} gives us
\[
\left|\frac{{\rm vol}(\hat{A}_{11})}{{\rm vol}(A_{11})}\right| = \left|1 - \frac{3}{2}\right|= \frac{1}{2}<1,
\]
where the $3/2$ factor comes from a diagonal entry of $A_{11}$. We conclude that $A_{11}$ is a local maximum volume submatrix of $A$. 

In~\cref{ex:sharpness}, we also needed to calculate $\sigma_1(A-A_k)$ and bound $\sigma_{k+1}(A)$. We can compute $\sigma_1(A-A_k)$ from the relationship $\sigma_1(A-A_k) = \sigma_1(S(A_{11}))$, where $S(A_{11})$ is the Schur complement given in~\cref{eq.rankkLU}. By a calculation, we find that 
\[
S(A_{11}) = (k+1)\mathfrak{1}_{m-k,n-k} - \mathfrak{1}_{m-k,k}A_{11}^{-1}\mathfrak{1}_{k,n-k} = \frac{k+2}{2}\mathfrak{1}_{m-k,n-k}, 
\]
where $\mathfrak{1}_{m,n}$ is an $m\times n$ matrix of all ones. Hence, we find that $\sigma_1(A-A_k) = (k+2)\sqrt{(m-k)(n-k)}/2$. To bound $\sigma_{k+1}(A)$ from above, we consider the following rank $\leq k$ matrix given by 
\[
   B_k \!=\!\! \left[\begin{array}{ccccc|ccc}
        k+1 & \!-1-\beta & \cdots & \!-1-\beta & \!-1-\beta & -1 & \cdots & -1 \\
       -1-\beta & k+1  & \cdots & \!-1-\beta & \!-1-\beta & -1 & \cdots & -1 \\
        \vdots & \vdots & \ddots & \vdots & \vdots & \vdots & \ddots & \vdots \\
        -1-\beta & \!-1-\beta & \cdots & k+1  & \!-1-\beta & -1 & \cdots & -1\\
        -1-\beta & \!-1-\beta & \cdots &\!-1-\beta & k+1  & -1 & \cdots & -1 \\
        \hline
        -1 & -1 & \cdots & -1 & -1 & k+1  & \cdots & k+1  \\
        \vdots & \vdots & \ddots & \vdots & \vdots & \vdots & \ddots & \vdots \\
       -1 & -1 & \cdots & -1 & -1 & k+1  & \cdots & k+1 
    \end{array}\right]\!, \quad \beta = \frac{k+2}{k^2-1}. 
\]
To see that $B_k$ is of rank $\leq k$, one can note that there are $k+1$ distinct columns and $k+1$ distinct rows. Moreover, the sum of the first $k$ columns (or rows) equals $k/(k+1)$ times the $(k+1)$st column (or row). The entries of $A$ and $B_k$ match everywhere except for $k(k-1)$ entries in the principal $k\times k$ submatrix. This means by the Eckart--Young Theorem~\cite{eckart1936approximation} that for $k\geq 2$ we have
\[
\sigma_{k+1}(A)\leq \|A-B_k\|_2 \leq \|A-B_k\|_F = \sqrt{k(k-1)\beta^2} = \sqrt{\frac{k(k-1)(k+2)^2}{(k+1)^2(k-1)^2}}\leq 2,
\]
where $\|\cdot\|_F$ denotes the matrix Frobenius norm.


\section*{Acknowledgments}
During the pandemic, A.T. co-organized the Complexity of Matrix Computations (CMC) online seminar series. The seminar on ``RRQR and column subset selection" on the 10th of November 2020 significantly motivated us. We are grateful to Ilse Ipsen for discussing local maximum volume pivoting with A.T. over Zoom in 2021. A.D., A.T. and A.Y. were supported by the SciAI Center, funded by the Office of Naval Research under Grant Number N00014-23-1-2729. A.D. was supported by the National Science Foundation under award DMS-2146079. A.T. was supported by NSF CAREER (DMS-2045646). 

\bibliographystyle{siam}
\bibliography{references}

\end{document}